\documentclass[reqno]{amsart}

\usepackage{amssymb}
\usepackage{graphicx}
\usepackage{amscd}
\usepackage[hidelinks]{hyperref}
\usepackage{color}
\usepackage{float}
\usepackage{graphics,amsmath,amssymb}
\usepackage{amsthm}
\usepackage{amsfonts}
\usepackage{latexsym}
\usepackage{epsf}
\usepackage{enumitem}
\usepackage{xifthen}
\usepackage{mathrsfs}
\usepackage{dsfont}
\usepackage{makecell}
\usepackage[FIGTOPCAP]{subfigure}
\usepackage{amsmath}
\allowdisplaybreaks[4]
\usepackage{listings}
\usepackage{etoolbox}
\usepackage{fancyhdr}
\usepackage{pdflscape}
\usepackage[title,toc,titletoc]{appendix}
\usepackage{framed}
\usepackage{blkarray}
\usepackage[noadjust]{cite}

 \newtheoremstyle{mytheorem}
 {3pt}
 {3pt}
 {\slshape}
 {}
 {\bfseries}
 {.}
 { }
 {}

\numberwithin{equation}{section}

\theoremstyle{theorem}
\newtheorem{theorem}{Theorem}[section]

\newtheorem{lemma}[theorem]{Lemma}
\newtheorem{proposition}[theorem]{Proposition}
\newtheorem{claim}[theorem]{Claim}

\theoremstyle{definition}
\newtheorem{definition}{Definition}[section]

\newtheorem{example}{Example}[section]
\newtheorem*{example*}{Example}
\newtheorem{conjecture}{Conjecture}[section]

\theoremstyle{remark}
\newtheorem{remark}{Remark}[section]
\newtheorem*{remark*}{Remark}
\newtheorem*{remarks*}{Remarks}

\newtheorem{question}{Question}[section]

\newcommand{\Keywords}[1]{\ifthenelse{\isempty{#1}}{}{\smallskip \smallskip \noindent \textbf{Keywords}. #1}}
\newcommand{\MSC}[2][2010]{\ifthenelse{\isempty{#2}}{}{\smallskip \smallskip \noindent \textbf{#1MSC}. #2}}
\newcommand{\abstractnote}[1]{\ifthenelse{\isempty{#1}}{}{\smallskip \smallskip \noindent \textsuperscript{\dag}#1}}

\makeatletter
\def\specialsection{\@startsection{section}{1}%
  \z@{\linespacing\@plus\linespacing}{.5\linespacing}%
  {\normalfont}}
\def\section{\@startsection{section}{1}%
  \z@{.7\linespacing\@plus\linespacing}{.5\linespacing}%
  {\normalfont\scshape}}
\patchcmd{\@settitle}{\uppercasenonmath\@title}{\Large\boldmath}{}{}
\patchcmd{\@settitle}{\begin{center}}{\begin{flushleft}}{}{}
\patchcmd{\@settitle}{\end{center}}{\end{flushleft}}{}{}
\patchcmd{\@setauthors}{\MakeUppercase}{\normalsize}{}{}
\patchcmd{\@setauthors}{\centering}{\raggedright}{}{}
\patchcmd{\section}{\scshape}{\large\bfseries\boldmath}{}{}
\patchcmd{\subsection}{\bfseries}{\bfseries\boldmath}{}{}
\renewcommand{\@secnumfont}{\bfseries}
\patchcmd{\@startsection}{\@afterindenttrue}{\@afterindentfalse}{}{}
\patchcmd{\abstract}{\leftmargin3pc}{\leftmargin1pc}{}{}

\def\maketitle{\par
  \@topnum\z@ 
  \@setcopyright
  \thispagestyle{empty}
  \ifx\@empty\shortauthors \let\shortauthors\shorttitle
  \else \andify\shortauthors
  \fi
  \@maketitle@hook
  \begingroup
  \@maketitle
  \toks@\@xp{\shortauthors}\@temptokena\@xp{\shorttitle}%
  \toks4{\def\\{ \ignorespaces}}
  \edef\@tempa{%
    \@nx\markboth{\the\toks4
      \@nx\MakeUppercase{\the\toks@}}{\the\@temptokena}}%
  \@tempa
  \endgroup
  \c@footnote\z@
  \@cleartopmattertags
}

\def\maketag@@@#1{\hbox{\m@th\normalfont\normalsize#1}}
\makeatother

\newcommand{\ps}{\mathscr{P}}
\newcommand{\is}{\mathscr{I}}
\newcommand{\tail}{\mathrm{Tail}}
\newcommand{\tP}{\tilde{P}}
\newcommand{\tp}{\tilde{p}}

\title{Linked partition ideals and Kanade--Russell conjectures}

\author[S. Chern]{Shane Chern}
\address[S. Chern]{Department of Mathematics, The Pennsylvania State University, University Park, PA 16802, USA}
\email{shanechern@psu.edu}

\author[Z. Li]{Zhitai Li}
\address[Z. Li]{Department of Mathematics, The Pennsylvania State University, University Park, PA 16802, USA}
\email{zfl5082@psu.edu}

\date{}

\begin{document}

\maketitle


\begin{abstract}

This paper will primarily present a method of proving generating function identities for partitions from linked partition ideals. The method we introduce is built on a conjecture by George Andrews and that those generating functions satisfy some $q$-difference equations. We will come up with the generating functions of partitions in the Kanade--Russell conjectures to illustrate the effectiveness of this method.


\MSC{Primary 05A17; Secondary 11P84.}
\end{abstract}

\section{Introduction}

\subsection{Background}

As usual, a \textit{partition} $\lambda$ of a positive integer $n$ is a weakly decreasing sequence of positive integers $\lambda_1\ge \lambda_2\ge\cdots\ge \lambda_{\ell}$ whose sum equals $n$. By convention, we may assume that $0$ has one partition, which is called the empty partition $\emptyset$.

In the theory of partitions, generating function identities are of great interest as they encapsulate considerable information of the partitions in question. In a series of papers \cite{And1972,And1974,And1975} dated back to the 1970s, George Andrews initiated a general theory of partition identities. These papers were later included in Chapter 8 of Andrews' monograph ``\textit{The theory of partitions}'' \cite{And1976}. In particular, Andrews introduced the concept of linked partition ideal. Recently, Andrews further communicated the idea that linked partition ideals deserve some more attention for their generating functions can be elegantly formulated.

The following conjecture by Andrews provides us a basis of ``guessing'' the generating function: 
\begin{conjecture}[Andrews]
Every linked partition ideal has a bivariate generating function of the form
\begin{equation}\label{eq:And-conj-0}
\sum_{n_1,\ldots,n_r\ge 0}\frac{(-1)^{L_1(n_1,\ldots,n_r)}q^{Q(n_1,\ldots,n_r)+L_2(n_1,\ldots,n_r)}x^{L_3(n_1,\ldots,n_r)}}{(q^{B_1};q^{A_1})_{n_1}\cdots (q^{B_r};q^{A_r})_{n_r}},
\end{equation}
in which $L_1$, $L_2$ and $L_3$ are linear forms in $n_1,\ldots,n_r$ and $Q$ is a quadratic form in $n_1,\ldots,n_r$. Here the coefficient of the $x^mq^n$ term is the number of partitions of $n$ in this linked partition ideal with exactly $m$ parts.
\end{conjecture}

This conjecture has numerous pieces of empirical evidence:

\begin{enumerate}[label*=\arabic*.,leftmargin=*]
\item The first Rogers--Ramanujan identity (cf.~Corollary 7.6 in \cite{And1976}) states that the number of partitions of a nonnegative integer $n$ into parts congruent to $\pm 1$ modulo $5$ is the same as the number of partitions of $n$ such that each two consecutive parts have difference at least $2$. We know that the generating function of partitions under the above difference-at-a-distance theme is
\[
\sum_{n\ge 0}\frac{q^{n^2}}{(q;q)_n}.
\]
A generalization of the Rogers--Ramanujan identities is due to Gordon (cf.~Theorem 7.5 in \cite{And1976}). In a special case of Gordon's generalization, we encounter partitions of the form $\lambda_1+ \lambda_2+\cdots+ \lambda_{\ell}$, where for all $j$, $\lambda_j - \lambda_{j+k-1}\ge  2$ with $k\ge 2$ fixed. It can be shown that the generating function is
\[
\sum_{n_1,n_2,\ldots,n_{k-1}\ge 0}\frac{q^{N_1^2+N_2^2+\cdots+N_{k-1}^2}}{(q;q)_{n_1}(q;q)_{n_2}\cdots(q;q)_{n_{k-1}}},
\]
where $N_j=n_j+n_{j+1}+\cdots+n_{k-1}$. Andrews showed that this partition set is a linked partition ideal; see \cite[Chapter 8]{And1976}.

\item In the first G\"{o}llnitz--Gordon identity, one studies partitions of the form $\lambda_1+ \lambda_2+\cdots+ \lambda_{\ell}$, in which no odd part is repeated, $\lambda_j-\lambda_{j-1}\ge 2$ if $\lambda_j$ odd and $\lambda_j-\lambda_{j-1}> 2$ if $\lambda_j$ even. It can be shown that the generating function is
\[
\sum_{n_1,n_2,n_3\ge 0}\frac{(-1)^{n_2}q^{n_1^2+n_3^2+2n_1n_2+n_2}}{(q^2;q^2)_{n_1}(q^2;q^2)_{n_2}(q^2;q^2)_{n_3}}.
\]
This partition set is also a linked partition ideal as claimed by Andrews in \cite[Chapter 8]{And1976}.
\end{enumerate}

With the aid of the above conjecture and necessary computer algebra assistance, if we want to find a generating function identity for a linked partition ideal, we are able to single out the promising candidates by running through a number of multi-summations in the above fashion and comparing the series expansions.

\subsection{Kanade--Russell conjectures}

As we have already seen, many linked partition ideals consist of partitions under certain difference-at-a-distance theme. 

\begin{definition}[cf.~\cite{KR2015}]
	We say a partition $\lambda=\lambda_1+ \lambda_2+\cdots+ \lambda_{\ell}$ satisfies the \textit{difference at least $d$ at distance $k$} condition if, for all $j$, $\lambda_j - \lambda_{j+k}\ge  d$.
\end{definition}

In this setting, we may restate the corresponding partition set in the first Rogers--Ramanujan identity as ``the set of partitions with difference at least $2$ at distance $1$''.

In 2014, Kanade and Russell \cite{KR2015} proposed six challenging conjectures, labeled as $I_1$--$I_6$, on partition identities of Rogers--Ramanujan type. For example, the first of their conjectures reads as follows.

\begin{conjecture}[Kanade--Russell Conjecture $I_1$]
The number of partitions of a nonnegative integer $n$ into parts congruent to $1$, $3$, $6$ or $8$ modulo $9$ is the same as the number of partitions of $n$ with difference at least $3$ at distance $2$ such that if two consecutive parts differ by at most $1$, then their sum is divisible by $3$.
\end{conjecture}

It should be remarked that these conjectures are intriguingly related to the representation theory of affine Lie algebra. For a detailed description of the idea behind them, one may refer to Kanade's Ph.D. Thesis \cite{Kan2015}.

On the other hand, in Russell's Ph.D. Thesis \cite{Rus2016}, companions to the Kanade--Russell conjectures $I_4$--$I_6$ were considered. Further, several more conjectures of the same flavor were proposed in \cite{KR2018}. In particular, among these conjectures (including the six conjectures in \cite{KR2015}), there are eleven of them involving the modulus $12$. It is notable that in a very recent paper of Bringmann, Jennings-Shaffer and Mahlburg \cite{BJM2018}, seven of the modulo $12$ conjectures were proved, while the rest were, although not completely proved, simplified to a great extent.

One major difficulty of proving the Kanade--Russell conjectures is that it is not always easy to find generating functions for partitions under certain difference-at-a-distance themes. Fortunately, this problem was settled in two recent papers of Kanade and Russell \cite{KR2018}, and Kur\c{s}ung\"{o}z \cite{Kur2018}, in which different sets of identities (but with some overlap) were demonstrated, respectively. However, their proofs, although different, are both purely combinatorial.

However, if we notice that the partition sets under difference-at-a-distance themes in the six conjectures $I_1$--$I_6$ are either a linked partition ideal or a subset of a linked partition ideal, a more algebraic approach can be provided. The resulting generating functions, in turn, give us more evidence for Andrews' conjecture.

\subsection{Outline of the paper}

In Section \ref{sect:linked-partition-ideal}, we will give a detailed account of linked partition ideals and their generating functions. Among those, the most consequential property is that those generating functions satisfy some $q$-difference equations. Such a $q$-difference equation is obtained by solving a $q$-difference system. To do so, we reformulate in Section \ref{sec:q-diff} an algorithm due to Andrews (cf.~\cite[Lemma 8.10]{And1976}) into the matrix form to make it easier to manipulate in computer algebra systems like \textit{Mathematica}.

In the next three sections, as experiments, we apply our method to not only reprove the six generating function identities involved in the Kanade--Russell conjectures but also present six more new identities. It is notable that our method is also applicable to the cases where the partitions in question are from a nice subset of a linked partition ideal.

We end our paper with several interesting transformation formulas motivated by a recent paper of Bringmann, Jennings-Shaffer and Mahlburg \cite{BJM2018}.

\begin{remark}
The proof of Andrews' conjecture should be regarded as a worthy objective: if we could successfully prove this conjecture, we essentially obtain a universal and robust method of deducing generating functions for partitions from linked partition ideals.
\end{remark}

\section{Linked partition ideals}\label{sect:linked-partition-ideal}

We now give a brief review of linked partition ideals. Note that we shall restate some definitions in \cite[Chapter 8]{And1976}.

Let $\ps$ be the set of partitions. Given a partition $\lambda\in\ps$, let $|\lambda|$ denote the sum of all parts of $\lambda$, let $\sharp(\lambda)$ denote the number of parts in $\lambda$ and let $\sharp_k(\lambda)$ denote the number of occurrences of parts of size $k$ in $\lambda$. For example, if $\lambda=3+3+2+1+1+1$, then $\sharp(\lambda)=6$, $\sharp_1(\lambda)=3$, $\sharp_2(\lambda)=1$, $\sharp_3(\lambda)=2$ and $\sharp_k(\lambda)=0$ for $k\ge 4$. From the definition of partitions, one can see that only finitely many of the $\sharp_k(\lambda)$ are nonzero.

We define a partial order ``$\le$'' by asserting that, for any two partitions $\lambda$ and $\pi$, $\pi\le \lambda$ if $\sharp_k(\pi)\le \sharp_k(\lambda)$ for all $k$. Andrews also defined the ``meet'' and ``join'' operations for $\lambda$ and $\pi$ by treating $\ps$ as a lattice:
\begin{enumerate}
\item $\lambda\cap\pi$ satisfies $\sharp_k(\lambda\cap\pi)=\min(\sharp_k(\lambda), \sharp_k(\pi))$ for all $k$;
\item $\lambda\cup\pi$ satisfies $\sharp_k(\lambda\cup\pi)=\max(\sharp_k(\lambda), \sharp_k(\pi))$ for all $k$.
\end{enumerate}

\begin{definition}
A subset $\is$ of $\ps$ is called a \textit{partition ideal} if for any $\lambda$ in $\is$, $\pi$ is also in $\is$ whenever $\pi\le\lambda$.
\end{definition}

\begin{remark}
Andrews further asserted that a partition ideal is indeed a semi-ideal in the notation of lattice theory.
\end{remark}

We next define the modulus of a partition ideal. To do so, we need the following notation.

Let $\is$ be a partition ideal. We define $\is^{(m)}$ by the collection of partitions in $\is$ whose smallest part is $>m$. We also include the empty partition $\emptyset$ in $\is^{(m)}$.

We then define a map $\phi$ by sending a partition $\lambda=\lambda_1+\lambda_2+\cdots+\lambda_\ell$ to $(\lambda_1+1)+(\lambda_2+1)+\cdots+(\lambda_\ell+1)$ and the empty partition to itself.

\begin{definition}
We say that a partition ideal $\is$ has \textit{modulus} $m$ if $m$ is a positive integer such that $\phi^m \is=\is^{(m)}$.
\end{definition}

\begin{remark}
A partition ideal might have more than one modulus.
\end{remark}

For two partitions $\lambda$ and $\pi$ in $\is$, their sum $\lambda\oplus \pi$ is defined by collecting their parts in weakly decreasing order. Lemma 8.9 in \cite{And1976} gives a unique decomposition for each $\lambda\in\is$ if $\is$ has modulus $m$.

\begin{lemma}
Let $\is$ be a partition ideal of modulus $m$. For each $\lambda\in\is$, we uniquely have
$$\lambda=\lambda_{(1)}\oplus(\phi^m\lambda_{(2)})\oplus(\phi^{2m}\lambda_{(3)})\oplus\cdots$$
where $\lambda_{(1)}$, $\lambda_{(2)}$, $\lambda_{(3)}$, $\ldots$ are in $\is$, all satisfying the property that the largest part $\le m$.
\end{lemma}

\begin{definition}
We define, for each partition ideal $\is$ of modulus $m$,
$$L_{\is,m}:=\{\lambda\in \is: \text{the largest part of $\lambda\le m$}\}.$$
Here, again, the empty partition is included in $L_{\is,m}$. Further,  if the value of $m$ is clear from the context, then the $m$ in the subscript could be omitted and we simply write $L_{\is}:=L_{\is,m}$.
\end{definition}

\begin{definition}
For any partition $\pi\in\ps$, its \textit{$m$-tail} $\tail_m(\pi)$ is defined to be the collection of parts of $\pi$ which are at most $m$. For example,
$$\tail_2(3+3+2+1+1+1)=2+1+1+1.$$
\end{definition}

Now we are ready to give the definition of linked partition ideals.

\begin{definition}\label{def:linked}
We say that a partition ideal $\is$ is a \textit{linked partition ideal} if
\begin{enumerate}[label=(\roman*)]
\item $\is$ has a modulus, say $m$;
\item the $L_{\is}$ corresponding to $m$ is a finite set;
\item for each $\pi\in L_{\is}$, there corresponds a minimal subset $\mathcal{L}_{\is}(\pi)\subseteq L_{\is}$ (called the linking set of $\pi$) and a positive integer $l(\pi)$ (called the span of $\pi$) such that for any partition $\lambda$, it belongs to $\is$ with $\tail_m(\lambda)=\pi$ if and only if we can find a partition $\tilde{\pi}$ with $\tail_m(\tilde{\pi})\in \mathcal{L}_{\is}(\pi)$ such that
$$\lambda=\pi\oplus \left(\phi^{l(\pi)m} \tilde{\pi}\right).$$
\end{enumerate}
\end{definition}

\begin{example}
	Let $\mathscr{D}$ denote the set of partitions into distinct parts. A trivial observation is that $\mathscr{D}$ is a partition ideal of modulus $1$. Further, $L_{\mathscr{D}}=L_{\mathscr{D},1}=\{\emptyset,1\}$. We also have
	\begin{gather*}
	\mathcal{L}_{\mathscr{D}}(\emptyset)=\{\emptyset,1\},\quad l(\emptyset)=1,\\
	\intertext{and}
	\mathcal{L}_{\mathscr{D}}(1)=\{\emptyset,1\},\quad l(1)=1.
	\end{gather*}
\end{example}

\begin{example}
	Let $\mathscr{R}$ denote the set of partitions in which each two consecutive parts have difference at least $2$. We can find that $\mathscr{R}$ is a partition ideal of modulus $2$. Further, $L_{\mathscr{R}}=L_{\mathscr{R},2}=\{\emptyset,1,2\}$. We also have
	\begin{gather*}
	\mathcal{L}_{\mathscr{R}}(\emptyset)=\{\emptyset,1,2\},\quad l(\emptyset)=1,\\
	\mathcal{L}_{\mathscr{R}}(1)=\{\emptyset,1,2\},\quad l(1)=1,\\
	\intertext{and}
	\mathcal{L}_{\mathscr{R}}(2)=\{\emptyset,2\},\quad l(2)=1.
	\end{gather*}
\end{example}

Finally, we consider a bivariate generating function for any subset $\mathscr{S}$ of $\ps$:
\begin{equation}
G_{\mathscr{S}}(x)=G_{\mathscr{S}}(x,q):=\sum_{\lambda\in\mathscr{S}}x^{\sharp(\lambda)}q^{|\lambda|}.
\end{equation}
In the setting of Definition \ref{def:linked}, if we further define $\is_{\pi}$ by the set of partitions in $\is$ whose $m$-tail is $\pi\in L_{\is}$, then (8.4.13) in \cite{And1976} tells us that
\begin{equation}\label{eq:important-gf-1}
\sum_{\mu\in \is_{\pi}} x^{\sharp(\mu)} q^{|\mu|}= x^{\sharp(\pi)}q^{|\pi|} \sum_{\varpi\in \mathcal{L}_{\is}(\pi)}\sum_{\nu\in \is_{\varpi}} \left(xq^{l(\pi)m}\right)^{\sharp(\nu)} q^{|\nu|}.
\end{equation}
In other words,
\begin{equation}\label{eq:important-gf-2}
G_{\is_{\pi}}(x)= x^{\sharp(\pi)}q^{|\pi|} \sum_{\varpi\in \mathcal{L}_{\is}(\pi)}G_{\is_{\varpi}}(xq^{l(\pi)m}).
\end{equation}

\section{Systems of $q$-difference equations}\label{sec:q-diff}

As we will see in the next section, a crucial point there can be summarized as the following question:

\begin{question}
Suppose we have a system of $q$-difference equations, say,
\begin{equation}\label{eq:system-gen}
\left\{\!\!\begin{array}{c}
F_1(x)=p_{1,1}(x)F_1(xq^m)+p_{1,2}(x)F_2(xq^m)+\cdots+p_{1,k}(x)F_k(xq^m)\\[0.5em]
F_2(x)=p_{2,1}(x)F_1(xq^m)+p_{2,2}(x)F_2(xq^m)+\cdots+p_{2,k}(x)F_k(xq^m)\\[0.5em]
\vdots\\[0.5em]
F_k(x)=p_{k,1}(x)F_1(xq^m)+p_{k,2}(x)F_2(xq^m)+\cdots+p_{k,k}(x)F_k(xq^m)
\end{array}\right.,
\end{equation}
where the $F$'s and $p$'s are in $x$ and $q$, is it possible to deduce a $q$-difference equation merely involving $F_1$?
\end{question}

Fortunately, an affirmative algorithm is provided by Andrews in the proof of \cite[Lemma 8.10]{And1976}. We would like to translate Andrews' algorithm to the matrix form to make it more transparent.

\medskip

At first, the system \eqref{eq:system-gen} can be written in the matrix form
\begin{equation}\label{eq:step-0}
\begin{pmatrix}
F_1(x)\\
F_2(x)\\
\vdots\\
F_k(x)
\end{pmatrix}
=
\begin{pmatrix}
p_{1,1}(x) & p_{1,2}(x) & \cdots & p_{1,k}(x)\\
p_{2,1}(x) & p_{2,2}(x) & \cdots & p_{2,k}(x)\\
\vdots & \vdots & \ddots & \vdots\\
p_{k,1}(x) & p_{k,2}(x) & \cdots & p_{k,k}(x)\\
\end{pmatrix}
\begin{pmatrix}
F_1(xq^m)\\
F_2(xq^m)\\
\vdots\\
F_k(xq^m)
\end{pmatrix}.
\end{equation}

\medskip

\textbf{Step (1).} We put $u_1(x)=F_1(x)$. Then \eqref{eq:step-0} becomes
\begin{equation}\label{eq:step-1}
\begin{pmatrix}
u_1(x)\\
F_2(x)\\
\vdots\\
F_k(x)
\end{pmatrix}
=
\begin{pmatrix}
p_{1,1}(x) & p_{1,2}(x) & \cdots & p_{1,k}(x)\\
p_{2,1}(x) & p_{2,2}(x) & \cdots & p_{2,k}(x)\\
\vdots & \vdots & \ddots & \vdots\\
p_{k,1}(x) & p_{k,2}(x) & \cdots & p_{k,k}(x)\\
\end{pmatrix}
\begin{pmatrix}
u_1(xq^m)\\
F_2(xq^m)\\
\vdots\\
F_k(xq^m)
\end{pmatrix}.
\end{equation}

If $p_{1,2}(x)=p_{1,3}(x)=\cdots=p_{1,k}(x)=0$, then we shall terminate at this place by noticing that
\[
u_1(x)=p_{1,1}(x)u_1(xq^m).
\]

\medskip

For Steps ($\mathrm{s}$) with $2\le s\le k$, we proceed iteratively as follows.

\medskip

\textbf{Step ($\mathbf{s}$).} Supposing that in the $(s-1)$-th Step, we obtain
\begin{equation}\label{eq:step-i-1}
\begin{pmatrix}
u_1(x)\\
\vdots\\
u_{s-1}(x)\\
F_s(x)\\
\vdots\\
F_k(x)
\end{pmatrix}
=
\tP_{s-1}
\begin{pmatrix}
u_1(xq^m)\\
\vdots\\
u_{s-1}(xq^m)\\
F_s(xq^m)\\
\vdots\\
F_k(xq^m)
\end{pmatrix},
\end{equation}
where $\tP_{s-1}$ is a $k\times k$ matrix with the $(i,j)$-th entry being $\tp_{i,j}(x)$.

Since we have arrived at the $s$th Step, we know that at least one of the $\tp_{s-1,s}(x)$, $\tp_{s-1,s+1}(x)$, \ldots, $\tp_{s-1,k}(x)$ is not identically zero. Otherwise, the program should be terminated at the $(s-1)$-th Step. Further, if $\tp_{s-1,s}(x)$ is identically zero and $\tp_{s-1,t}(x)$ (for some $t$ with $s+1\le t\le k$) is not identically zero, \eqref{eq:step-i-1} can be rewritten by swapping $F_s$ and $F_t$. In such a case, $\tP_{s-1}$ should be rewritten by swapping $\tp_{s,s}(x)$ and $\tp_{t,t}(x)$, swapping $\tp_{s,t}(x)$ and $\tp_{t,s}(x)$, swapping $\tp_{i,s}(x)$ and $\tp_{i,t}(x)$ for $i\ne s,t$, and swapping $\tp_{s,j}(x)$ and $\tp_{t,j}(x)$ for $j\ne s,t$. For notational convenience, we simply rename $F_s$ by $F_t$ and $F_t$ by $F_s$ so that the new relation is still of the form \eqref{eq:step-i-1} while $\tp_{s-1,s}(x)$ is not identically zero.

We then make the following substitution
\begin{equation}
u_{s}(xq^m)=\tp_{s-1,s}(x)F_s(xq^m)+\tp_{s-1,s+1}(x)F_{s+1}(xq^m)+\cdots+\tp_{s-1,k}(x)F_k(xq^m).
\end{equation}
Written in the matrix form, we have
\begin{equation}\label{eq:step-i-trans}
\begin{pmatrix}
u_1(xq^m)\\
u_2(xq^m)\\
\vdots\\
u_{s-1}(xq^m)\\
u_{s}(xq^m)\\
F_{s+1}(xq^m)\\
\vdots\\
F_k(xq^m)
\end{pmatrix}
=
T(x)
\begin{pmatrix}
u_1(xq^m)\\
u_2(xq^m)\\
\vdots\\
u_{s-1}(xq^m)\\
F_s(xq^m)\\
F_{s+1}(xq^m)\\
\vdots\\
F_k(xq^m)
\end{pmatrix},
\end{equation}
where
$$T(x)=\begin{pmatrix}
1 & 0 & \cdots & 0 & 0 & 0 & \cdots & 0\\
0 & 1 & \cdots & 0 & 0 & 0 & \cdots & 0\\
\vdots & \vdots & \ddots & \vdots & \vdots & \vdots & \ddots & \vdots\\
0 & 0 & \cdots & 1 & 0 & 0 & \cdots & 0\\
0 & 0 & \cdots & 0 & \tp_{s-1,s}(x) & \tp_{s-1,s+1}(x) & \cdots & \tp_{s-1,k}(x)\\
0 & 0 & \cdots & 0 & 0 & 1 & \cdots & 0\\
\vdots & \vdots & \ddots & \vdots & \vdots & \vdots & \ddots & \vdots\\
0 & 0 & \cdots & 0 & 0 & 0 & \cdots & 1\\
\end{pmatrix}.$$
Here all diagonal entries in the $k\times k$ matrix $T(x)$ are $1$ except for the $s$th diagonal entry. In the $s$th row of $T(x)$, for $s\le t\le k$, the $(s,t)$-th entry is $\tp_{s-1,t}(x)$. All remaining entries in $T(x)$ are $0$.

Since $\tp_{s-1,s}(x)$ is not identically zero, the matrix $T(x)$ is invertible. In particular, we have
$$T(x)^{-1}=
\begin{pmatrix}
1 & 0 & \cdots & 0 & 0 & 0 & \cdots & 0\\
0 & 1 & \cdots & 0 & 0 & 0 & \cdots & 0\\
\vdots & \vdots & \ddots & \vdots & \vdots & \vdots & \ddots & \vdots\\
0 & 0 & \cdots & 1 & 0 & 0 & \cdots & 0\\
0 & 0 & \cdots & 0 & \frac{1}{\tp_{s-1,s}(x)} & -\frac{\tp_{s-1,s+1}(x)}{\tp_{s-1,s}(x)} & \cdots & -\frac{\tp_{s-1,k}(x)}{\tp_{s-1,s}(x)}\\
0 & 0 & \cdots & 0 & 0 & 1 & \cdots & 0\\
\vdots & \vdots & \ddots & \vdots & \vdots & \vdots & \ddots & \vdots\\
0 & 0 & \cdots & 0 & 0 & 0 & \cdots & 1\\
\end{pmatrix}.$$

It follows from \eqref{eq:step-i-1} and \eqref{eq:step-i-trans} that
\begin{equation}\label{eq:step-i-final}
\begin{pmatrix}
u_1(x)\\
\vdots\\
u_{s}(x)\\
F_{s+1}(x)\\
\vdots\\
F_k(x)
\end{pmatrix}
=
\tP_{s}
\begin{pmatrix}
u_1(xq^m)\\
\vdots\\
u_{s}(xq^m)\\
F_{s+1}(xq^m)\\
\vdots\\
F_k(xq^m)
\end{pmatrix},
\end{equation}
where
$$\tP_{s}=T(xq^{-m}) \tP_{s-1} T(x)^{-1}.$$

\begin{claim}\label{claim:q-diff}
The matrix $\tP_{s}$ obtained above is of the form
\[
\begin{blockarray}{rccccccccc}
& 1 & 2 & 3 & 4 & \cdots & s & s+1 & \cdots & k \\
\begin{block}{r(ccccccccc)}
1 &  \star & 1 & 0 & 0 & \cdots & 0 & 0 & \cdots & 0 \\
2 & \star & \star & 1 & 0 & \cdots & 0 & 0 & \cdots & 0 \\
\vdots &  \vdots & \vdots &\vdots &\vdots &\ddots &\vdots &\vdots &\ddots&\vdots   \\
s-1 &  \star & \star & \star & \star & \cdots & 1 & 0 & \cdots & 0 \\
s &  \star & \star & \star & \star & \cdots & \star & \star & \cdots & \star \\
\vdots &  \vdots & \vdots &\vdots &\vdots &\ddots &\vdots &\vdots &\ddots&\vdots   \\
k &  \star & \star & \star & \star & \cdots & \star & \star & \cdots & \star \\
\end{block}
\end{blockarray}\ .
 \]
 More precisely, in row $r$ $(1\le r\le s-1)$ of $\tP_{s}$, the $(r,r+1)$-th entry is $1$ and the $(r,c)$-th entries are $0$ for all $c>r+1$.
\end{claim}

\begin{proof}
We argue by induction on $s$. When $s=1$, there is nothing to prove. Assuming that the result is true for some $s-1$ and noticing that $\tP_{s-1}$ is such a matrix obtained in the $(s-1)$-th Step, we know that $\tp_{r,r+1}(x)=1$ for all $r\le s-2$ and that $\tp_{r,c}(x)=0$ for all $r\le s-2$ and $c>r+1$.

It is obvious that the first $s-1$ rows of $T(xq^{-m}) \tP_{s-1}$ are identical to the first $s-1$ rows of $\tP_{s-1}$. Let the $(j,c)$-th entry of $T(x)^{-1}$ be $T_{j,c}^{(-1)}(x)$. 

For $r\le s-1$, the $(r,c)$-th entry of $\tP_{s}=T(xq^{-m}) \tP_{s-1} T(x)^{-1}$ is given by
$$\sum_{j=1}^k \tp_{r,j}(x)T_{j,c}^{(-1)}(x).$$

If $c=r+1$, then the only non-zero contribution in the above summation is
\begin{align*}
\tp_{r,r+1}(x)T_{r+1,r+1}^{(-1)}(x)&=\begin{cases}
1\cdot 1 & \text{if $r\le s-2$}\\
\tp_{s-1,s}(x)\cdot \frac{1}{\tp_{s-1,s}(x)} & \text{if $r=s-1$}
\end{cases}\\[0.5em]
&=1.
\end{align*}

If $c>r+1$, then we first treat the $r=s-1$ case. One has
\begin{align*}
\sum_{j=1}^k \tp_{s-1,j}(x)T_{j,c}^{(-1)}(x)&=\tp_{s-1,s}(x)T_{s,c}^{(-1)}(x)+\tp_{s-1,c}(x)T_{c,c}^{(-1)}(x)\\
&=\tp_{s-1,s}(x)\cdot \left(-\frac{\tp_{s-1,c}(x)}{\tp_{s-1,s}(x)}\right)+\tp_{s-1,c}(x)\cdot 1\\
&=0.
\end{align*}
For $r\le s-2$, we simply notice that $\tp_{r,j}(x)=0$ for $j>r+1$ from our assumption and that $T_{j,c}^{(-1)}(x)=0$ for $j\le r+1$ since $j\le s-1$ and $j\ne c$.
\end{proof}

Let $\tp_{i,j}^{\text{New}}(x)$ be the $(i,j)$-th entry of $\tP_{s}$. If $\tp_{s,t}^{\text{New}}(x)=0$ for all $t\ge s+1$, then we shall stop at this place by noticing with the help of Claim \ref{claim:q-diff} that
\begin{align*}
u_1(x)&=\tp_{1,1}^{\text{New}}(x)u_1(xq^m)+u_2(xq^m),\\[0.5em]
u_2(x)&=\tp_{2,1}^{\text{New}}(x)u_1(xq^m)+\tp_{2,2}^{\text{New}}(x)u_2(xq^m)+u_3(xq^m),\\[0.5em]
&\;\;\vdots\\[0.5em]
u_{s-1}(x)&=\tp_{s-1,1}^{\text{New}}(x)u_1(xq^m)+\tp_{s-1,2}^{\text{New}}(x)u_2(xq^m)+\cdots+u_{s}(xq^m),\\[0.5em]
u_{s}(x)&=\tp_{s,1}^{\text{New}}(x)u_1(xq^m)+\tp_{s,2}^{\text{New}}(x)u_2(xq^m)+\cdots+\tp_{s,s}^{\text{New}}(x)u_{s}(xq^m).
\end{align*}

\textbf{Final setup.} Assuming that the above program is terminated after $\ell$ ($\le k$) steps, we obtain a new system of $q$-difference equations
{\small\begin{align*}
u_1(x)&=r_{1,1}(x)u_1(xq^m)+u_2(xq^m),\\[0.5em]
u_2(x)&=r_{2,1}(x)u_1(xq^m)+r_{2,2}(x)u_2(xq^m)+u_3(xq^m),\\[0.5em]
&\;\;\vdots\\[0.5em]
u_{\ell-1}(x)&=r_{\ell-1,1}(x)u_1(xq^m)+r_{\ell-1,2}(x)u_2(xq^m)+\cdots+r_{\ell-1,\ell-1}(x)u_{\ell-1}(x)+u_{\ell}(xq^m),\\[0.5em]
u_{\ell}(x)&=r_{\ell,1}(x)u_1(xq^m)+r_{\ell,2}(x)u_2(xq^m)+\cdots+r_{\ell,\ell-1}(x)u_{\ell-1}(xq^m)+r_{\ell,\ell}(x)u_{\ell}(xq^m),
\end{align*}}
where the $r$'s are in $x$ and $q$.

With this new system, a $q$-difference equation involving merely $u_1$ can be obtained by simple eliminations. Finally, we recall that $F_1(x)$ is set to be $u_1(x)$ in Step (1).

\section{Kanade--Russell conjectures}

We may summarize the following four types of partition sets under difference-at-a-distance themes from the Kanade--Russell conjectures.

\smallskip

\noindent\textbullet{~}\;\textsc{Type I}:

Partitions with difference at least $3$ at distance $2$ such that if two consecutive parts differ by at most $1$, then their sum is divisible by $3$.

\smallskip

\noindent\textbullet{~}\;\textsc{Type II}:

Partitions with difference at least $3$ at distance $2$ such that if two consecutive parts differ by at most $1$, then their sum is congruent to $2$ modulo $3$.

\smallskip

\noindent\textbullet{~}\;\textsc{Type III}:

Partitions with difference at least $3$ at distance $3$ such that if parts at distance $2$ differ by at most $1$, then the sum of the two parts and their intermediate part is congruent to $1$ modulo $3$.

\smallskip

\noindent\textbullet{~}\;\textsc{Type IV}:

Partitions with difference at least $3$ at distance $3$ such that if parts at distance $2$ differ by at most $1$, then the sum of the two parts and their intermediate part is congruent to $2$ modulo $3$.

\medskip

In this section, we investigate partition sets of types I, II, III and IV in the setting of linked partition ideals.

\subsection{Partition set of type I}

Recall that the partition set of type I is the set of partitions with difference at least $3$ at distance $2$ such that if two consecutive parts differ by at most $1$, then their sum is divisible by $3$. In other words, if $\lambda=\lambda_1+\lambda_2+\cdots+\lambda_\ell$ is in this partition set, then
\begin{enumerate}[label=(\roman*)]
\item $\lambda_i-\lambda_{i+2}\ge 3$;
\item $\lambda_i-\lambda_{i+1}\le 1$ implies $\lambda_i+\lambda_{i+1}\equiv 0 \pmod{3}$.
\end{enumerate}

Let $\ps_{T_{\mathrm{I}}}$ denote the partition set of type I.

\begin{claim}
$\ps_{T_{\mathrm{I}}}$ is a partition ideal of modulus $3$.
\end{claim}

\begin{proof}
We first prove that $\ps_{T_{\mathrm{I}}}$ is a partition ideal. It suffices to show that for any $\lambda=\lambda_1+\lambda_2+\cdots+\lambda_\ell$ in $\ps_{T_{\mathrm{I}}}$, if we delete a part from $\lambda$, the resulting partition $\tilde{\lambda}$ is still in $\ps_{T_{\mathrm{I}}}$. Obviously, if the deleted part is $\lambda_1$ or $\lambda_\ell$, then $\tilde{\lambda}\in\ps_{T_{\mathrm{I}}}$. Hence, we may assume that $\tilde{\lambda}=\lambda_1+\cdots+\lambda_{k-1}+\lambda_{k+1}+\cdots+\lambda_\ell$ with $\lambda_k$ not being $\lambda_1$ or $\lambda_\ell$. To see that $\tilde{\lambda}$ satisfies the first condition, it suffices to check that $\lambda_{k-2}-\lambda_{k+1}\ge \lambda_{k-1}-\lambda_{k+1}\ge 3$ and $\lambda_{k-1}-\lambda_{k+2}\ge \lambda_{k-1}-\lambda_{k+1}\ge 3$. On the other hand, $\lambda_{k-1}-\lambda_{k+1}\ge 3$ and the fact that $\lambda\in\ps_{T_{\mathrm{I}}}$ ensure the second condition. Hence $\tilde{\lambda}\in\ps_{T_{\mathrm{I}}}$, as desired. Finally, the fact that $\ps_{T_{\mathrm{I}}}$ has modulus $3$ is trivial.
\end{proof}

From the definition of $\ps_{T_{\mathrm{I}}}$, it is straightforward to observe the following facts.

\begin{claim}
The $L_{\ps_{T_{\mathrm{I}}}}$ corresponding to modulus $3$ equals
$$\{\emptyset,\;1,\;2+1,\;3+1,\;2,\;3,\;3+3\}.$$
\end{claim}

\begin{claim}\label{claim:T1-3}
The span and linking set of partitions in $L_{\ps_{T_{\mathrm{I}}}}$ are given as follows.
\begin{equation*}
\begin{array}{lp{0.5cm}cp{0.5cm}l}
&& \text{span} && \quad\quad\quad\quad\text{linking set}\\
\pi_0 = \emptyset && 1 && \{\pi_0,\; \pi_1,\; \pi_2,\; \pi_3,\; \pi_4,\; \pi_5,\; \pi_6\}\\
\pi_1 = 1 && 1 && \{\pi_0,\; \pi_1,\; \pi_2,\; \pi_3,\; \pi_4,\; \pi_5,\; \pi_6\}\\
\pi_2 = 2+1 && 1 && \{\pi_0,\; \pi_1,\; \pi_2,\; \pi_3,\; \pi_4,\; \pi_5,\; \pi_6\}\\
\pi_3 = 3+1 && 1 && \{\pi_0,\; \pi_4,\; \pi_5,\; \pi_6\}\\
\pi_4 = 2 && 1 && \{\pi_0,\; \pi_1,\; \pi_2,\; \pi_3,\; \pi_4,\; \pi_5,\; \pi_6\}\\
\pi_5 = 3 && 1 && \{\pi_0,\; \pi_4,\; \pi_5,\; \pi_6\}\\
\pi_6 = 3+3 && 1 && \{\pi_0,\; \pi_5,\; \pi_6\}
\end{array}
\end{equation*}
\end{claim}

Let us denote by $H_i(x)=H_i(x,q)$ the generating function of partitions $\lambda$ in $\ps_{T_{\mathrm{I}}}$ with $\tail_m(\lambda)=\pi_i$ for $i=0,1,\ldots, 6$ where the $\pi_i$'s are as defined in Claim \ref{claim:T1-3}.

Following \eqref{eq:important-gf-2}, we have
{\footnotesize\begin{align}
H_0(x)&=H_0(xq^3)+H_1(xq^3)+H_2(xq^3)+H_3(xq^3)+H_4(xq^3)+H_5(xq^3)+H_6(xq^3),\notag\\
\label{eq:T1H0}\\
x^{-1}q^{-1}H_1(x)&=H_0(xq^3)+H_1(xq^3)+H_2(xq^3)+H_3(xq^3)+H_4(xq^3)+H_5(xq^3)+H_6(xq^3),\notag\\
\label{eq:T1H1}\\
x^{-2}q^{-3}H_2(x)&=H_0(xq^3)+H_1(xq^3)+H_2(xq^3)+H_3(xq^3)+H_4(xq^3)+H_5(xq^3)+H_6(xq^3),\notag\\
\label{eq:T1H2}\\
x^{-2}q^{-4}H_3(x)&=H_0(xq^3)+H_4(xq^3)+H_5(xq^3)+H_6(xq^3),\notag\\
\label{eq:T1H3}\\
x^{-1}q^{-2}H_4(x)&=H_0(xq^3)+H_1(xq^3)+H_2(xq^3)+H_3(xq^3)+H_4(xq^3)+H_5(xq^3)+H_6(xq^3),\notag\\
\label{eq:T1H4}\\
x^{-1}q^{-3}H_5(x)&=H_0(xq^3)+H_4(xq^3)+H_5(xq^3)+H_6(xq^3),\notag\\
\label{eq:T1H5}\\
x^{-2}q^{-6}H_6(x)&=H_0(xq^3)+H_5(xq^3)+H_6(xq^3).\notag\label{eq:T1H6}\\
\end{align}}

Let $G_{\ps_{T_{\mathrm{I},1}}}(x)=G_{\ps_{T_{\mathrm{I},1}}}(x,q)$ (resp.~$G_{\ps_{T_{\mathrm{I},2}}}(x)$, $G_{\ps_{T_{\mathrm{I},3}}}(x)$) denote the generating function of partitions in $\ps_{T_{\mathrm{I}}}$ whose smallest part is at least $1$ (resp.~$2$, $3$).

It follows that
\begin{align}
G_{\ps_{T_{\mathrm{I},1}}}(x)&=H_0(x)+H_1(x)+H_2(x)+H_3(x)+H_4(x)+H_5(x)+H_6(x)\nonumber\\
&=H_0(xq^{-3}),\label{eq:GT11}\\
G_{\ps_{T_{\mathrm{I},2}}}(x)&=H_0(x)+H_4(x)+H_5(x)+H_6(x)\nonumber\\
&=x^{-1}H_5(xq^{-3}),\label{eq:GT12}\\
G_{\ps_{T_{\mathrm{I},3}}}(x)&=H_0(x)+H_5(x)+H_6(x)\nonumber\\
&=x^{-2}H_6(xq^{-3}).\label{eq:GT13}
\end{align}
Hence, to determine $q$-difference equations satisfied by $G_{\ps_{T_{\mathrm{I},1}}}(x)$, $G_{\ps_{T_{\mathrm{I},2}}}(x)$ and $G_{\ps_{T_{\mathrm{I},3}}}(x)$, it suffices to find $q$-difference equations for $H_0(x)$, $H_5(x)$ and $H_6(x)$, respectively.

We now deduce from \eqref{eq:T1H0}, \eqref{eq:T1H1}, \eqref{eq:T1H2} and \eqref{eq:T1H4} that
\begin{align}
H_1(x)&=xqH_0(x),\\
H_2(x)&=x^2 q^3 H_0(x),\\
H_4(x)&=x q^2 H_0(x),\\
\intertext{and likewise from \eqref{eq:T1H3} and \eqref{eq:T1H5} that}
H_3(x)&=xqH_5(x).
\end{align}

As a result, the system \eqref{eq:T1H0}--\eqref{eq:T1H6} can be rewritten as
\begin{framed}
{\small\begin{align}
H_0(x)&=(1+xq^4+x^2q^9+xq^5)H_0(xq^3)+(1+xq^4)H_5(xq^3)+H_6(xq^3),\label{eq:T1H0New}\\
H_5(x)&=(xq^3+x^2q^8)H_0(xq^3)+xq^3H_5(xq^3)+xq^3H_6(xq^3),\label{eq:T1H5New}\\
H_6(x)&=x^2q^6H_0(xq^3)+x^2q^6H_5(xq^3)+x^2q^6H_6(xq^3).\label{eq:T1H6New}
\end{align}}
\end{framed}

We first use the algorithm in Section \ref{sec:q-diff} to deduce the $q$-difference equation satisfied by $H_0(x)$ and accordingly $G_{\ps_{T_{\mathrm{I},1}}}(x)$.

\medskip

\textbf{Step (1).} We put $u_0(x)=H_0(x)$. Then

\begin{equation}
\begin{pmatrix}
u_0(x)\\
H_5(x)\\
H_6(x)
\end{pmatrix}
=\tP_1
\begin{pmatrix}
u_0(xq^3)\\
H_5(xq^3)\\
H_6(xq^3)
\end{pmatrix},
\end{equation}
where
\[
\tP_1=\begin{pmatrix}
1+xq^4+x^2q^9+xq^5 & 1+xq^4 & 1\\
xq^3+x^2q^8 & xq^3 & xq^3\\
x^2q^6 & x^2q^6 & x^2q^6
\end{pmatrix}.
\]

\textbf{Step (2).} We put $u_5(x)=(1+xq^4)H_5(xq^3)+H_6(xq^3)$. Then

\begin{equation}
\begin{pmatrix}
u_0(x)\\
u_5(x)\\
H_6(x)
\end{pmatrix}
=
\tP_2
\begin{pmatrix}
u_0(xq^3)\\
u_5(xq^3)\\
H_6(xq^3)
\end{pmatrix},
\end{equation}
where
\[
\tP_2=\begin{pmatrix}
1+xq^4+xq^5+x^2q^9 & 1 & 0\\[0.5em]
xq^3(1+xq+xq^3+xq^5+x^2q^6) & \frac{xq^3(1+xq+xq^3)}{1+xq^4} & \frac{x^2q^7(1+xq+xq^3)}{1+xq^4}\\[0.5em]
x^2q^6 & \frac{x^2q^6}{1+xq^4} & \frac{x^3q^{10}}{1+xq^4}
\end{pmatrix}.
\]

\textbf{Step (3).} We put $u_6(x)=\frac{x^2q^7(1+xq+xq^3)}{1+xq^4}H_6(xq^3)$. Then

\begin{equation}
\begin{pmatrix}
u_0(x)\\
u_5(x)\\
u_6(x)
\end{pmatrix}
=
\tP_3
\begin{pmatrix}
u_0(xq^3)\\
u_5(xq^3)\\
u_6(xq^3)
\end{pmatrix},
\end{equation}
where
\[
\tP_3=\begin{pmatrix}
1+xq^4+xq^5+x^2q^9 & 1 & 0\\[0.5em]
xq^3(1+xq+xq^3+xq^5+x^2q^6) & \frac{xq^3(1+xq+xq^3)}{1+xq^4} & 1\\[0.5em]
\frac{x^4q^7(1+x+xq^{-2})}{1+xq} & \frac{x^4q^7(1+x+xq^{-2})}{(1+xq)(1+xq^4)} & \frac{x^3q^4(1+x+xq^{-2})}{(1+xq)(1+xq+xq^3)}
\end{pmatrix}.
\]

\medskip

For convenience, we write
\begin{align}
u_0(x)&=r_{0,0}(x)u_0(xq^3)+u_5(xq^3),\label{eq:T1NewNewNew0}\\
u_5(x)&=r_{5,0}(x)u_0(xq^3)+r_{5,5}(x)u_5(xq^3)+u_6(xq^3),\label{eq:T1NewNewNew5}\\
u_6(x)&=r_{6,0}(x)u_0(xq^3)+r_{6,5}(x)u_5(xq^3)+r_{6,6}u_6(xq^3),\label{eq:T1NewNewNew6}
\end{align}
where the coefficients are rational functions in $x$ and $q$ given by $\tP_3$.

Noting from \eqref{eq:GT11} that
\begin{align}\label{eq:u0-eli}
G_{\ps_{T_{\mathrm{I},1}}}(x)=H_0(xq^{-3})=u_0(xq^{-3}),
\end{align}
we may eliminate $u_5(x)$ by \eqref{eq:T1NewNewNew0},
\begin{align}\label{eq:u5-eli}
u_5(x)=G_{\ps_{T_{\mathrm{I},1}}}(x)-r_{0,0}(xq^{-3}) G_{\ps_{T_{\mathrm{I},1}}}(xq^3).
\end{align}
Substituting \eqref{eq:u5-eli} into \eqref{eq:T1NewNewNew5}, we may eliminate $u_6(x)$,
\begin{align}
u_6(x)&=G_{\ps_{T_{\mathrm{I},1}}}(xq^{-3})-\left(r_{0,0}(xq^{-6})+r_{5,5}(xq^{-3})\right) G_{\ps_{T_{\mathrm{I},1}}}(x)\nonumber\\
&\quad+ \left(r_{0,0}(xq^{-3})r_{5,5}(xq^{-3})-r_{5,0}(xq^{-3})\right)G_{\ps_{T_{\mathrm{I},1}}}(xq^3).\label{eq:u6-eli}
\end{align}

Substituting \eqref{eq:u0-eli}, \eqref{eq:u5-eli} and \eqref{eq:u6-eli} into \eqref{eq:T1NewNewNew6}, we arrive at, after simplification, the following $q$-difference equation for $G_{\ps_{T_{\mathrm{I},1}}}(x)$.

\begin{theorem}\label{th:q-diff-T11}
It holds that
\begin{align}
p_0(x,q)G_{\ps_{T_{\mathrm{I},1}}}(x)+p_3(x,q)G_{\ps_{T_{\mathrm{I},1}}}(xq^3)&+p_6(x,q)G_{\ps_{T_{\mathrm{I},1}}}(x q^6)\nonumber\\
&+p_9(x,q)G_{\ps_{T_{\mathrm{I},1}}}(xq^9)=0,\label{eq:q-diff-T11}
\end{align}
where
\begin{align*}
p_0(x,q)&=1+x(q^4+q^6),\\
p_3(x,q)&=-1-x(q + q^2 + q^3 + q^4 + q^6)-x^2(q^3 + q^4 + q^5 + 2 q^6 + q^7 + q^8 + q^9)\\
&\quad-x^3(q^7 + q^9 + q^{10} + q^{12}),\\
p_6(x,q)&=x^3(q^{11} + q^{13})+x^4(q^{14} + q^{15} + q^{16} + q^{17} + q^{18})+x^5(q^{19} + q^{21}),\\
\intertext{and}
p_9(x,q)&=x^5 q^{27}+x^6(q^{28}+q^{30}).
\end{align*}
\end{theorem}

In the same manner, we may find the $q$-difference equations for $H_5(x)$ and $H_6(x)$, and accordingly $G_{\ps_{T_{\mathrm{I},2}}}(x)$ and $G_{\ps_{T_{\mathrm{I},3}}}(x)$.

\begin{theorem}
It holds that
\begin{align}
p_0(x,q)G_{\ps_{T_{\mathrm{I},2}}}(x)+p_3(x,q)G_{\ps_{T_{\mathrm{I},2}}}(xq^3)&+p_6(x,q)G_{\ps_{T_{\mathrm{I},2}}}(x q^6)\nonumber\\
&+p_9(x,q)G_{\ps_{T_{\mathrm{I},2}}}(xq^9)=0,
\end{align}
where
\begin{align*}
p_0(x,q)&=1+x(q^5+q^8),\\
p_3(x,q)&=-1-x(q^2 + q^3 + q^4 + q^5 + q^8)\nonumber\\
&\quad-x^2(2q^6 + q^7 + q^8 + q^9 + q^{10} + q^{11} + q^{12})-x^3(q^{11} + 2 q^{14} + q^{17}),\\
p_6(x,q)&=x^3(q^{16} + q^{17})+x^4(-q^{17} + q^{18} + q^{19} + q^{21} + q^{22} + q^{23} + q^{24})\\
&\quad+x^5(q^{26} + q^{29}),\\
\intertext{and}
p_9(x,q)&=x^5 q^{33}+x^6(q^{35}+q^{38}).
\end{align*}
\end{theorem}

\begin{theorem}
It holds that
\begin{align}
p_0(x,q)G_{\ps_{T_{\mathrm{I},3}}}(x)+p_3(x,q)G_{\ps_{T_{\mathrm{I},3}}}(xq^3)&+p_6(x,q)G_{\ps_{T_{\mathrm{I},3}}}(x q^6)\nonumber\\
&+p_9(x,q)G_{\ps_{T_{\mathrm{I},3}}}(xq^9)=0,
\end{align}
where
\begin{align*}
p_0(x,q)&=1+x(q^6+q^7),\\
p_3(x,q)&=-1-x(q^3 + q^4 + q^5 + q^6 + q^7)\\
&\quad-x^2(q^6 + q^8 + 2q^9 + 2q^{10} + q^{11} + q^{12})-x^3(q^{12} + q^{13} + q^{15} + q^{16}),\\
p_6(x,q)&=x^3(q^{16} + q^{17})+x^4(q^{20} + q^{21} + q^{22} + q^{23} + q^{24})+x^5(q^{27} + q^{28}),\\
\intertext{and}
p_9(x,q)&=x^5 q^{36}+x^6(q^{39}+q^{40}).
\end{align*}
\end{theorem}

\begin{remark}
	A simple \textit{Mathematica} program is implemented to compute these recurrences. We are happy to supply our codes upon request.
\end{remark}

\subsection{Partition set of type II}

Recall that the partition set of type II is the set of partitions with difference at least $3$ at distance $2$ such that if two consecutive parts differ by at most $1$, then their sum is congruent to $2$ modulo $3$. In other words, if $\lambda=\lambda_1+\lambda_2+\cdots+\lambda_\ell$ is in this partition set, then
\begin{enumerate}[label=(\roman*)]
\item $\lambda_i-\lambda_{i+2}\ge 3$;
\item $\lambda_i-\lambda_{i+1}\le 1$ implies $\lambda_i+\lambda_{i+1}\equiv 2 \pmod{3}$.
\end{enumerate}

Let $\ps_{T_{\mathrm{II}}}$ denote the partition set of type II.

\begin{claim}
$\ps_{T_{\mathrm{II}}}$ is a partition ideal of modulus $3$.
\end{claim}

\begin{claim}
The $L_{\ps_{T_{\mathrm{II}}}}$ corresponding to modulus $3$ equals
$$\{\emptyset,\;1,\;1+1,\;3+1,\;2,\;3+2,\;3\}.$$
\end{claim}

\begin{claim}\label{claim:T2-3}
The span and linking set of partitions in $L_{\ps_{T_{\mathrm{II}}}}$ are given as follows.
\begin{equation*}
\begin{array}{lp{0.5cm}cp{0.5cm}l}
&& \text{span} && \quad\quad\quad\quad\text{linking set}\\
\pi_0 = \emptyset && 1 && \{\pi_0,\; \pi_1,\; \pi_2,\; \pi_3,\; \pi_4,\; \pi_5,\; \pi_6\}\\
\pi_1 = 1 && 1 && \{\pi_0,\; \pi_1,\; \pi_2,\; \pi_3,\; \pi_4,\; \pi_5,\; \pi_6\}\\
\pi_2 = 1+1 && 1 && \{\pi_0,\; \pi_1,\; \pi_2,\; \pi_3,\; \pi_4,\; \pi_5,\; \pi_6\}\\
\pi_3 = 3+1 && 1 && \{\pi_0,\; \pi_4,\; \pi_5,\; \pi_6\}\\
\pi_4 = 2 && 1 && \{\pi_0,\; \pi_1,\; \pi_3,\; \pi_4,\; \pi_5,\; \pi_6\}\\
\pi_5 = 3+2 && 1 && \{\pi_0,\; \pi_4,\; \pi_5,\; \pi_6\}\\
\pi_6 = 3 && 1 && \{\pi_0,\; \pi_4,\; \pi_5,\; \pi_6\}
\end{array}
\end{equation*}
\end{claim}

Similarly, let us denote by $H_i(x)=H_i(x,q)$ the generating function of partitions $\lambda$ in $\ps_{T_{\mathrm{II}}}$ with $\tail_m(\lambda)=\pi_i$ for $i=0,1,\ldots, 6$ where the $\pi_i$'s are as defined in Claim \ref{claim:T2-3}.

Following \eqref{eq:important-gf-2}, we have
{\footnotesize\begin{align}
H_0(x)&=H_0(xq^3)+H_1(xq^3)+H_2(xq^3)+H_3(xq^3)+H_4(xq^3)+H_5(xq^3)+H_6(xq^3),\notag\\
\label{eq:T2H0}\\
x^{-1}q^{-1}H_1(x)&=H_0(xq^3)+H_1(xq^3)+H_2(xq^3)+H_3(xq^3)+H_4(xq^3)+H_5(xq^3)+H_6(xq^3),\notag\\
\label{eq:T2H1}\\
x^{-2}q^{-2}H_2(x)&=H_0(xq^3)+H_1(xq^3)+H_2(xq^3)+H_3(xq^3)+H_4(xq^3)+H_5(xq^3)+H_6(xq^3),\notag\\
\label{eq:T2H2}\\
x^{-2}q^{-4}H_3(x)&=H_0(xq^3)+H_4(xq^3)+H_5(xq^3)+H_6(xq^3),\notag\\
\label{eq:T2H3}\\
x^{-1}q^{-2}H_4(x)&=H_0(xq^3)+H_1(xq^3)+H_3(xq^3)+H_4(xq^3)+H_5(xq^3)+H_6(xq^3),\notag\\
\label{eq:T2H4}\\
x^{-2}q^{-5}H_5(x)&=H_0(xq^3)+H_4(xq^3)+H_5(xq^3)+H_6(xq^3),\notag\\
\label{eq:T2H5}\\
x^{-1}q^{-3}H_6(x)&=H_0(xq^3)+H_4(xq^3)+H_5(xq^3)+H_6(xq^3).\notag\label{eq:T2H6}\\
\end{align}}

Let $G_{\ps_{T_{\mathrm{II},1}}}(x)=G_{\ps_{T_{\mathrm{II},1}}}(x,q)$ (resp.~$G_{\ps_{T_{\mathrm{II},2}}}(x)$) denote the generating function of partitions in $\ps_{T_{\mathrm{II}}}$ whose smallest part is at least $1$ (resp.~$2$).

Let $G_{\ps_{T_{\mathrm{II},a}}}(x)$ denote the generating function of partitions in $\ps_{T_{\mathrm{II}}}$ where $1$ appears at most once.

It follows that
\begin{align}
G_{\ps_{T_{\mathrm{II},1}}}(x)&=H_0(x)+H_1(x)+H_2(x)+H_3(x)+H_4(x)+H_5(x)+H_6(x)\nonumber\\
&=H_0(xq^{-3}),\label{eq:GT21}\\
G_{\ps_{T_{\mathrm{II},2}}}(x)&=H_0(x)+H_4(x)+H_5(x)+H_6(x)\nonumber\\
&=x^{-1}H_6(xq^{-3}),\label{eq:GT22}\\
G_{\ps_{T_{\mathrm{II},a}}}(x)&=H_0(x)+H_1(x)+H_3(x)+H_4(x)+H_5(x)+H_6(x)\nonumber\\
&=x^{-1}qH_4(xq^{-3}).\label{eq:GT23}
\end{align}

We may deduce from \eqref{eq:T2H0}, \eqref{eq:T2H1} and \eqref{eq:T2H2} that
\begin{align}
H_1(x)&=xqH_0(x),\\
H_2(x)&=x^2 q^2 H_0(x),\\
\intertext{and likewise from \eqref{eq:T2H3}, \eqref{eq:T2H5} and \eqref{eq:T2H6} that}
H_3(x)&=xqH_6(x),\\
H_5(x)&=xq^2H_6(x).
\end{align}

Hence, the system \eqref{eq:T2H0}--\eqref{eq:T2H6} can be rewritten as
\begin{framed}
{\small\begin{align}
H_0(x)&=(1+xq^4+x^2q^8)H_0(xq^3)+H_4(xq^3)+(1+xq^4+xq^5)H_6(xq^3),\label{eq:T2H0New}\\
H_4(x)&=(x q^2+x^2 q^6)H_0(xq^3)+x q^2 H_4(xq^3)+(x q^2+x^2q^6+x^2 q^7)H_6(xq^3),\label{eq:T2H4New}\\
H_6(x)&=x q^3H_0(xq^3)+x q^3H_4(xq^3)+(x q^3+x^2q^8)H_6(xq^3).\label{eq:T2H6New}
\end{align}}
\end{framed}

Using the algorithm in Section \ref{sec:q-diff}, we are able to prove the following $q$-difference equations for $G_{\ps_{T_{\mathrm{II},1}}}(x)$, $G_{\ps_{T_{\mathrm{II},2}}}(x)$ and $G_{\ps_{T_{\mathrm{II},a}}}(x)$, respectively.

\begin{theorem}
It holds that
\begin{align}
p_0(x,q)G_{\ps_{T_{\mathrm{II},1}}}(x)+p_3(x,q)G_{\ps_{T_{\mathrm{II},1}}}(xq^3)&+p_6(x,q)G_{\ps_{T_{\mathrm{II},1}}}(x q^6)\nonumber\\
&+p_9(x,q)G_{\ps_{T_{\mathrm{II},1}}}(xq^9)=0,
\end{align}
where
\begin{align*}
p_0(x,q)&=1+x(q^4+q^5),\\
p_3(x,q)&=-1-x(q + q^2 + q^3 + q^4 + q^5)-x^2(q^2 + q^4 + 2 q^5 + 2 q^6 + q^7 + q^8)\\
&\quad-x^3(q^6 + q^7 + q^9 + q^{10}),\\
p_6(x,q)&=x^3(q^{10} + q^{11})+x^4(q^{12} + q^{13} + q^{14} + q^{15} + q^{16})+x^5(q^{17} + q^{18}),\\
\intertext{and}
p_9(x,q)&=x^5 q^{26}+x^6(q^{27}+q^{28}).
\end{align*}
\end{theorem}

\begin{theorem}
It holds that
\begin{align}
p_0(x,q)G_{\ps_{T_{\mathrm{II},2}}}(x)+p_3(x,q)G_{\ps_{T_{\mathrm{II},2}}}(xq^3)&+p_6(x,q)G_{\ps_{T_{\mathrm{II},2}}}(x q^6)\nonumber\\
&+p_9(x,q)G_{\ps_{T_{\mathrm{II},2}}}(xq^9)=0,
\end{align}
where
\begin{align*}
p_0(x,q)&=1+x(q^5+q^7),\\
p_3(x,q)&=-1-x(q^2 + q^3 + q^4 + q^5 + q^7)\\
&\quad-x^2(q^5 + q^6 + q^7 + 2 q^8 + q^9 + q^{10} + q^{11}) - x^3(q^{10} + q^{12} + q^{13} + q^{15}),\\
p_6(x,q)&=x^3(q^{14} + q^{16})+x^4(q^{18} + q^{19} + q^{20} + q^{21} + q^{22})+x^5(q^{24} + q^{26}),\\
\intertext{and}
p_9(x,q)&=x^5 q^{32}+x^6(q^{34}+q^{36}).
\end{align*}
\end{theorem}

\begin{theorem}
It holds that
\begin{align}
p_0(x,q)G_{\ps_{T_{\mathrm{II},a}}}(x)+p_3(x,q)G_{\ps_{T_{\mathrm{II},a}}}(xq^3)&+p_6(x,q)G_{\ps_{T_{\mathrm{II},a}}}(x q^6)\nonumber\\
&+p_9(x,q)G_{\ps_{T_{\mathrm{II},a}}}(xq^9)=0,
\end{align}
where
\begin{align*}
p_0(x,q)&=1+x(q^4+q^8),\\
p_3(x,q)&=-1-x(q + q^2 + q^3 + q^4 + q^8)\\
&\quad-x^2(q^4 + 2 q^5 + q^6 + q^8 + q^9 + q^{10} + q^{11})-x^3(q^9 + q^{12} + q^{13} + q^{16}),\\
p_6(x,q)&=x^3(-q^{12} + q^{13} + q^{14} + q^{15})\\
&\quad+x^4(-q^{13} + q^{15} + q^{16} + q^{19} + q^{20} + q^{21} + q^{22})+x^5(q^{23} + q^{27}),\\
\intertext{and}
p_9(x,q)&=x^5 q^{29}+x^6(q^{30}+q^{34}).
\end{align*}
\end{theorem}

\subsection{Partition set of type III}

Recall that the partition set of type III is the set of partitions with difference at least $3$ at distance $3$ such that if parts at distance $2$ differ by at most $1$, then the sum of the two parts and their intermediate part is congruent to $1$ modulo $3$. In other words, if $\lambda=\lambda_1+\lambda_2+\cdots+\lambda_\ell$ is in this partition set, then
\begin{enumerate}[label=(\roman*)]
\item $\lambda_i-\lambda_{i+3}\ge 3$;
\item $\lambda_i-\lambda_{i+2}\le 1$ implies $\lambda_i+\lambda_{i+1}+\lambda_{i+2}\equiv 1 \pmod{3}$.
\end{enumerate}

Let $\ps_{T_{\mathrm{III}}}$ denote the partition set of type III.

\begin{claim}
$\ps_{T_{\mathrm{III}}}$ is a partition ideal of modulus $3$.
\end{claim}

\begin{claim}
The $L_{\ps_{T_{\mathrm{III}}}}$ corresponding to modulus $3$ equals
\begin{align*}
\{&\emptyset,\;1,\;1+1,\;2,\;2+1,\;2+1+1,\;2+2,\;3,\;3+1,\\
&3+1+1,\;3+2,\;3+2+1,\;3+2+2,\;3+3,\;3+3+1\}.
\end{align*}
\end{claim}

\begin{claim}\label{claim:T3-3}
The span and linking set of partitions in $L_{\ps_{T_{\mathrm{III}}}}$ are given as follows.
{\footnotesize\begin{equation*}
\begin{array}{lcp{0.05cm}l}
& \text{span} && \quad\quad\quad\quad\quad\quad\quad\quad\quad\quad\quad\quad\text{linking set}\\
\pi_0 = \emptyset & 1 && \{\pi_0,\; \pi_1,\; \pi_2,\; \pi_3,\; \pi_4,\; \pi_5,\; \pi_6,\; \pi_7,\; \pi_8,\; \pi_9,\; \pi_{10},\; \pi_{11},\; \pi_{12},\; \pi_{13},\; \pi_{14}\}\\
\pi_1 = 1 & 1 && \{\pi_0,\; \pi_1,\; \pi_2,\; \pi_3,\; \pi_4,\; \pi_5,\; \pi_6,\; \pi_7,\; \pi_8,\; \pi_9,\; \pi_{10},\; \pi_{11},\; \pi_{12},\; \pi_{13},\; \pi_{14}\}\\
\pi_2 = 1+1 & 1 && \{\pi_0,\; \pi_1,\; \pi_2,\; \pi_3,\; \pi_4,\; \pi_5,\; \pi_6,\; \pi_7,\; \pi_8,\; \pi_9,\; \pi_{10},\; \pi_{11},\; \pi_{12},\; \pi_{13},\; \pi_{14}\}\\
\pi_3 = 2 & 1 && \{\pi_0,\; \pi_1,\; \pi_2,\; \pi_3,\; \pi_4,\; \pi_5,\; \pi_6,\; \pi_7,\; \pi_8,\; \pi_9,\; \pi_{10},\; \pi_{11},\; \pi_{12},\; \pi_{13},\; \pi_{14}\}\\
\pi_4 = 2+1 & 1 && \{\pi_0,\; \pi_1,\; \pi_2,\; \pi_3,\; \pi_4,\; \pi_5,\; \pi_6,\; \pi_7,\; \pi_8,\; \pi_9,\; \pi_{10},\; \pi_{11},\; \pi_{12},\; \pi_{13},\; \pi_{14}\}\\
\pi_5 = 2+1+1 & 1 && \{\pi_0,\; \pi_1,\; \pi_2,\; \pi_3,\; \pi_4,\; \pi_5,\; \pi_6,\; \pi_7,\; \pi_8,\; \pi_9,\; \pi_{10},\; \pi_{11},\; \pi_{12},\; \pi_{13},\; \pi_{14}\}\\
\pi_6 = 2+2 & 1 && \{\pi_0,\; \pi_1,\; \pi_3,\; \pi_4,\; \pi_6,\; \pi_7,\; \pi_8,\; \pi_{10},\; \pi_{11},\; \pi_{12},\; \pi_{13},\; \pi_{14}\}\\
\pi_7 = 3 & 1 && \{\pi_0,\; \pi_1,\; \pi_3,\; \pi_4,\; \pi_6,\; \pi_7,\; \pi_8,\; \pi_{10},\; \pi_{11},\; \pi_{12},\; \pi_{13},\; \pi_{14}\}\\
\pi_8 = 3+1 & 1 && \{\pi_0,\; \pi_1,\; \pi_3,\; \pi_4,\; \pi_6,\; \pi_7,\; \pi_8,\; \pi_{10},\; \pi_{11},\; \pi_{12},\; \pi_{13},\; \pi_{14}\}\\
\pi_9 = 3+1+1 & 1 && \{\pi_0,\; \pi_1,\; \pi_3,\; \pi_4,\; \pi_6,\; \pi_7,\; \pi_8,\; \pi_{10},\; \pi_{11},\; \pi_{12},\; \pi_{13},\; \pi_{14}\}\\
\pi_{10} = 3+2 & 1 && \{\pi_0,\; \pi_1,\; \pi_3,\; \pi_4,\; \pi_6,\; \pi_7,\; \pi_8,\; \pi_{10},\; \pi_{11},\; \pi_{12},\; \pi_{13},\; \pi_{14}\}\\
\pi_{11} = 3+2+1 & 1 && \{\pi_0,\; \pi_1,\; \pi_3,\; \pi_4,\; \pi_6,\; \pi_7,\; \pi_8,\; \pi_{10},\; \pi_{11},\; \pi_{12},\; \pi_{13},\; \pi_{14}\}\\
\pi_{12} = 3+2+2 & 1 && \{\pi_0,\; \pi_3,\; \pi_6,\; \pi_7,\; \pi_{10},\; \pi_{12},\; \pi_{13}\}\\
\pi_{13} = 3+3 & 1 && \{\pi_0,\; \pi_1,\; \pi_3,\; \pi_7,\; \pi_8,\; \pi_{10},\; \pi_{13},\; \pi_{14}\}\\
\pi_{14} = 3+3+1 & 1 && \{\pi_0,\; \pi_1,\; \pi_3,\; \pi_7,\; \pi_8,\; \pi_{10},\; \pi_{13},\; \pi_{14}\}\\
\end{array}
\end{equation*}}
\end{claim}

Let us denote by $H_i(x)=H_i(x,q)$ the generating function of partitions $\lambda$ in $\ps_{T_{\mathrm{III}}}$ with $\tail_m(\lambda)=\pi_i$ for $i=0,1,\ldots, 14$ where the $\pi_i$'s are as defined in Claim \ref{claim:T3-3}.

Following \eqref{eq:important-gf-2}, we have
{\footnotesize\begin{align}
H_0(x)&=H_0(xq^3)+H_1(xq^3)+H_2(xq^3)+H_3(xq^3)+H_4(xq^3)+H_5(xq^3)\notag\\
&\quad+H_6(xq^3)+H_7(xq^3)+H_8(xq^3)+H_9(xq^3)+H_{10}(xq^3)\notag\\
&\quad+H_{11}(xq^3)+H_{12}(xq^3)+H_{13}(xq^3)+H_{14}(xq^3),\\
x^{-1}q^{-1}H_1(x)&=H_0(xq^3)+H_1(xq^3)+H_2(xq^3)+H_3(xq^3)+H_4(xq^3)+H_5(xq^3)\notag\\
&\quad+H_6(xq^3)+H_7(xq^3)+H_8(xq^3)+H_9(xq^3)+H_{10}(xq^3)\notag\\
&\quad+H_{11}(xq^3)+H_{12}(xq^3)+H_{13}(xq^3)+H_{14}(xq^3),\\
x^{-2}q^{-2}H_2(x)&=H_0(xq^3)+H_1(xq^3)+H_2(xq^3)+H_3(xq^3)+H_4(xq^3)+H_5(xq^3)\notag\\
&\quad+H_6(xq^3)+H_7(xq^3)+H_8(xq^3)+H_9(xq^3)+H_{10}(xq^3)\notag\\
&\quad+H_{11}(xq^3)+H_{12}(xq^3)+H_{13}(xq^3)+H_{14}(xq^3),\\
x^{-1}q^{-2}H_3(x)&=H_0(xq^3)+H_1(xq^3)+H_2(xq^3)+H_3(xq^3)+H_4(xq^3)+H_5(xq^3)\notag\\
&\quad+H_6(xq^3)+H_7(xq^3)+H_8(xq^3)+H_9(xq^3)+H_{10}(xq^3)\notag\\
&\quad+H_{11}(xq^3)+H_{12}(xq^3)+H_{13}(xq^3)+H_{14}(xq^3),\\
x^{-2}q^{-3}H_4(x)&=H_0(xq^3)+H_1(xq^3)+H_2(xq^3)+H_3(xq^3)+H_4(xq^3)+H_5(xq^3)\notag\\
&\quad+H_6(xq^3)+H_7(xq^3)+H_8(xq^3)+H_9(xq^3)+H_{10}(xq^3)\notag\\
&\quad+H_{11}(xq^3)+H_{12}(xq^3)+H_{13}(xq^3)+H_{14}(xq^3),\\
x^{-3}q^{-4}H_5(x)&=H_0(xq^3)+H_1(xq^3)+H_2(xq^3)+H_3(xq^3)+H_4(xq^3)+H_5(xq^3)\notag\\
&\quad+H_6(xq^3)+H_7(xq^3)+H_8(xq^3)+H_9(xq^3)+H_{10}(xq^3)\notag\\
&\quad+H_{11}(xq^3)+H_{12}(xq^3)+H_{13}(xq^3)+H_{14}(xq^3),\\
x^{-2}q^{-4}H_6(x)&=H_0(xq^3)+H_1(xq^3)+H_3(xq^3)+H_4(xq^3)+H_6(xq^3)+H_7(xq^3)\notag\\
&\quad+H_8(xq^3)+H_{10}(xq^3)+H_{11}(xq^3)+H_{12}(xq^3)+H_{13}(xq^3)\notag\\
&\quad+H_{14}(xq^3),\\
x^{-1}q^{-3}H_7(x)&=H_0(xq^3)+H_1(xq^3)+H_3(xq^3)+H_4(xq^3)+H_6(xq^3)+H_7(xq^3)\notag\\
&\quad+H_8(xq^3)+H_{10}(xq^3)+H_{11}(xq^3)+H_{12}(xq^3)+H_{13}(xq^3)\notag\\
&\quad+H_{14}(xq^3),\\
x^{-2}q^{-4}H_8(x)&=H_0(xq^3)+H_1(xq^3)+H_3(xq^3)+H_4(xq^3)+H_6(xq^3)+H_7(xq^3)\notag\\
&\quad+H_8(xq^3)+H_{10}(xq^3)+H_{11}(xq^3)+H_{12}(xq^3)+H_{13}(xq^3)\notag\\
&\quad+H_{14}(xq^3),\\
x^{-3}q^{-5}H_9(x)&=H_0(xq^3)+H_1(xq^3)+H_3(xq^3)+H_4(xq^3)+H_6(xq^3)+H_7(xq^3)\notag\\
&\quad+H_8(xq^3)+H_{10}(xq^3)+H_{11}(xq^3)+H_{12}(xq^3)+H_{13}(xq^3)\notag\\
&\quad+H_{14}(xq^3),\\
x^{-2}q^{-5}H_{10}(x)&=H_0(xq^3)+H_1(xq^3)+H_3(xq^3)+H_4(xq^3)+H_6(xq^3)+H_7(xq^3)\notag\\
&\quad+H_8(xq^3)+H_{10}(xq^3)+H_{11}(xq^3)+H_{12}(xq^3)+H_{13}(xq^3)\notag\\
&\quad+H_{14}(xq^3),\\
x^{-3}q^{-6}H_{11}(x)&=H_0(xq^3)+H_1(xq^3)+H_3(xq^3)+H_4(xq^3)+H_6(xq^3)+H_7(xq^3)\notag\\
&\quad+H_8(xq^3)+H_{10}(xq^3)+H_{11}(xq^3)+H_{12}(xq^3)+H_{13}(xq^3)\notag\\
&\quad+H_{14}(xq^3),\\
x^{-3}q^{-7}H_{12}(x)&=H_0(xq^3)+H_3(xq^3)+H_6(xq^3)+H_7(xq^3)+H_{10}(xq^3)+H_{12}(xq^3)\notag\\
&\quad+H_{13}(xq^3),\\
x^{-2}q^{-6}H_{13}(x)&=H_0(xq^3)+H_1(xq^3)+H_3(xq^3)+H_7(xq^3)+H_8(xq^3)+H_{10}(xq^3)\notag\\
&\quad+H_{13}(xq^3)+H_{14}(xq^3),\\
x^{-3}q^{-7}H_{14}(x)&=H_0(xq^3)+H_1(xq^3)+H_3(xq^3)+H_7(xq^3)+H_8(xq^3)+H_{10}(xq^3)\notag\\
&\quad+H_{13}(xq^3)+H_{14}(xq^3).
\end{align}}

This system may be simplified as
\begin{framed}
{\begin{align}
H_0(x)&=(1+xq^4+x^2q^8+xq^5+x^2q^9+x^3q^{13})H_0(xq^3)\notag\\
&\quad+(xq^4+1+xq^4+x^2q^8+xq^5+x^2q^9)H_7(xq^3)\notag\\
&\quad+H_{12}(xq^3)+(1+xq^4)H_{13}(xq^3),\\
H_7(x)&=(xq^3+x^2q^7+x^2q^8+x^3q^{12})H_0(xq^3)\notag\\
&\quad+(x^2q^7+xq^3+x^2q^7+x^2q^8+x^3q^{12})H_7(xq^3)\notag\\
&\quad+xq^3H_{12}(xq^3)+(xq^3+x^2q^7)H_{13}(xq^3),\\
H_{12}(x)&=(x^3q^7+x^4q^{12})H_0(xq^3)+(x^4q^{11}+x^3q^7+x^4q^{12})H_7(xq^3)\notag\\
&\quad+x^3q^7H_{12}(xq^3)+x^3q^7H_{13}(xq^3),\\
H_{13}(x)&=(x^2q^6+x^3q^{10}+x^3q^{11})H_0(xq^3)+(x^2q^6+x^3q^{10}+x^3q^{11})H_7(xq^3)\notag\\
&\quad+(x^2q^6+x^3q^{10})H_{13}(xq^3).
\end{align}}
\end{framed}

Let $G_{\ps_{T_{\mathrm{III},1}}}(x)=G_{\ps_{T_{\mathrm{III},1}}}(x,q)$ (resp.~$G_{\ps_{T_{\mathrm{III},2}}}(x)$) denote the generating function of partitions in $\ps_{T_{\mathrm{III}}}$ whose smallest part is at least $1$ (resp.~$2$).

Let $G_{\ps_{T_{\mathrm{III},a}}}(x)$ denote the generating function of partitions in $\ps_{T_{\mathrm{III}}}$ where $1$ appears at most once.

It follows that
\begin{align}
G_{\ps_{T_{\mathrm{III},1}}}(x)&=H_0(x)+H_1(x)+H_2(x)+H_3(x)+H_4(x)+H_5(x)\notag\\
&\quad+H_6(x)+H_7(x)+H_8(x)+H_9(x)+H_{10}(x)\notag\\
&\quad+H_{11}(x)+H_{12}(x)+H_{13}(x)+H_{14}(x)\notag\\
&=H_0(xq^{-3}),\\
G_{\ps_{T_{\mathrm{III},2}}}(x)&=H_0(x)+H_3(x)+H_6(x)+H_7(x)+H_{10}(x)+H_{12}(x)\notag\\
&\quad+H_{13}(x)\notag\\
&=x^{-3}q^2H_{12}(xq^{-3}),\\
G_{\ps_{T_{\mathrm{III},a}}}(x)&=H_0(x)+H_1(x)+H_3(x)+H_4(x)+H_6(x)+H_7(x)\notag\\
&\quad+H_8(x)+H_{10}(x)+H_{11}(x)+H_{12}(x)+H_{13}(x)\notag\\
&\quad+H_{14}(x)\notag\\
&=x^{-1}H_7(xq^{-3}).
\end{align}

Analogously, we can use the algorithm in Section \ref{sec:q-diff} to deduce the following $q$-difference equations for $G_{\ps_{T_{\mathrm{III},1}}}(x)$, $G_{\ps_{T_{\mathrm{III},2}}}(x)$ and $G_{\ps_{T_{\mathrm{III},a}}}(x)$, respectively.

\begin{theorem}
It holds that
\begin{align}
p_0(x,q)G_{\ps_{T_{\mathrm{III},1}}}(x)&+p_3(x,q)G_{\ps_{T_{\mathrm{III},1}}}(xq^3)+p_6(x,q)G_{\ps_{T_{\mathrm{III},1}}}(x q^6)\nonumber\\
&+p_9(x,q)G_{\ps_{T_{\mathrm{III},1}}}(xq^9)+p_{12}(x,q)G_{\ps_{T_{\mathrm{III},1}}}(xq^{12})=0,
\end{align}
where
{\footnotesize\begin{align*}
p_0(x,q)&=1+x(q^{4}+q^{5}+2 q^{7}+q^{9}+q^{10})\\
&\quad+ x^2(q^{9}+2 q^{11}+q^{12}+q^{13}+2 q^{14}+q^{15}+q^{16}+2 q^{17}+q^{19})\\
&\quad+ x^3(q^{16}+q^{18}+q^{19}+2 q^{21}+q^{23}+q^{24}+q^{26}),\\
p_3(x,q)&=-1-x(q+q^{2}+q^{3}+q^{4}+q^{5}+2 q^{7}+q^{9}+q^{10})\\
&\quad
-x^2(q^{2}+q^{3}+2 q^{4}+2 q^{5}+3 q^{6}+2 q^{7}+3 q^{8}+3 q^{9}+3 q^{10}+4 q^{11}+3 q^{12}+2 q^{13}\\
&\quad\quad\quad\quad+2 q^{14}+q^{15}+q^{16}+2 q^{17}+q^{19})\\
&\quad
-  x^3(q^{4}+q^{5}+2 q^{6}+4 q^{7}+3 q^{8}+5 q^{9}+5 q^{10}+6 q^{11}+7 q^{12}+8 q^{13}+7 q^{14}+6 q^{15}\\
&\quad\quad\quad\quad+6 q^{16}+4 q^{17}+5 q^{18}+4 q^{19}+3 q^{20}+3 q^{21}+q^{22}+q^{23}+q^{24}+q^{26})\\
&\quad
-x^4(q^{8}+2 q^{9}+2 q^{10}+5 q^{11}+4 q^{12}+6 q^{13}+10 q^{14}+8 q^{15}+10 q^{16}+11 q^{17}+8 q^{18}\\
&\quad\quad\quad\quad+10 q^{19}+9 q^{20}+8 q^{21}+7 q^{22}+6 q^{23}+4 q^{24}+3 q^{25}+2 q^{26}+2 q^{27}+q^{28}\\
&\quad\quad\quad\quad+q^{29})\\
&\quad
-x^5(q^{13}+3 q^{15}+3 q^{16}+4 q^{17}+7 q^{18}+6 q^{19}+7 q^{20}+10 q^{21}+7 q^{22}+9 q^{23}+9 q^{24}\\
&\quad\quad\quad\quad+6 q^{25}+7 q^{26}+5 q^{27}+4 q^{28}+3 q^{29}+3 q^{30}+q^{31}+q^{32})\\
&\quad
- x^6(q^{20}+2 q^{22}+2 q^{23}+q^{24}+4 q^{25}+2 q^{26}+3 q^{27}+4 q^{28}+2 q^{29}+3 q^{30}+3 q^{31}\\
&\quad\quad\quad\quad+q^{32}+2 q^{33}+q^{34}+q^{36}),\\
p_6(x,q)&=x^4(q^{12}+q^{14}+2 q^{16}+q^{18}+q^{20})\\
&\quad
+x^5(q^{13}+2 q^{15}+q^{16}+4 q^{17}+3 q^{18}+4 q^{19}+4 q^{20}+5 q^{21}+5 q^{22}+4 q^{23}+4 q^{24}\\
&\quad\quad\quad\quad+3 q^{25}+4 q^{26}+q^{27}+2 q^{28}+q^{30})\\
&\quad
+x^6(q^{17}+2 q^{18}+3 q^{19}+5 q^{20}+5 q^{21}+9 q^{22}+9 q^{23}+10 q^{24}+12 q^{25}+12 q^{26}\\
&\quad\quad\quad\quad+14 q^{27}+12 q^{28}+12 q^{29}+10 q^{30}+9 q^{31}+9 q^{32}+5 q^{33}+5 q^{34}+3 q^{35}\\
&\quad\quad\quad\quad+2 q^{36}+q^{37})\\
&\quad
+x^7(q^{22}+3 q^{23}+4 q^{24}+6 q^{25}+7 q^{26}+12 q^{27}+12 q^{28}+16 q^{29}+18 q^{30}+16 q^{31}\\
&\quad\quad\quad\quad+19 q^{32}+19 q^{33}+16 q^{34}+18 q^{35}+16 q^{36}+12 q^{37}+12 q^{38}+7 q^{39}+6 q^{40}\\
&\quad\quad\quad\quad+4 q^{41}+3 q^{42}+q^{43})\\
&\quad
+x^8(q^{28}+2 q^{29}+3 q^{30}+6 q^{31}+6 q^{32}+9 q^{33}+11 q^{34}+11 q^{35}+13 q^{36}+16 q^{37}\\
&\quad\quad\quad\quad+12 q^{38}+16 q^{39}+13 q^{40}+11 q^{41}+11 q^{42}+9 q^{43}+6 q^{44}+6 q^{45}+3 q^{46}\\
&\quad\quad\quad\quad+2 q^{47}+q^{48})\\
&\quad
+x^9(q^{35}+q^{36}+q^{37}+3 q^{38}+2 q^{39}+3 q^{40}+5 q^{41}+3 q^{42}+5 q^{43}+5 q^{44}+3 q^{45}\\
&\quad\quad\quad\quad+5 q^{46}+3 q^{47}+2 q^{48}+3 q^{49}+q^{50}+q^{51}+q^{52}),\\
p_{9}(x,q)&=-x^6q^{30}
-x^7(q^{31}+q^{32}+2 q^{34}+q^{36}+q^{37}+q^{38}+q^{39}+q^{40})\\
&\quad
-x^8(q^{33}+2 q^{35}+q^{36}+q^{37}+2 q^{38}+2 q^{39}+3 q^{40}+4 q^{41}+3 q^{42}+3 q^{43}+3 q^{44}\\
&\quad\quad\quad\quad+2 q^{45}+3 q^{46}+2 q^{47}+2 q^{48}+q^{49}+q^{50})\\
&\quad
-x^9(q^{37}+q^{39}+q^{40}+q^{41}+3 q^{42}+3 q^{43}+4 q^{44}+5 q^{45}+4 q^{46}+6 q^{47}+6 q^{48}\\
&\quad\quad\quad\quad+7 q^{49}+8 q^{50}+7 q^{51}+6 q^{52}+5 q^{53}+5 q^{54}+3 q^{55}+4 q^{56}+2 q^{57}+q^{58}\\
&\quad\quad\quad\quad+q^{59})\\
&\quad
-x^{10}(q^{45}+q^{46}+2 q^{47}+2 q^{48}+3 q^{49}+4 q^{50}+6 q^{51}+7 q^{52}+8 q^{53}+9 q^{54}+10 q^{55}\\
&\quad\quad\quad\quad+8 q^{56}+11 q^{57}+10 q^{58}+8 q^{59}+10 q^{60}+6 q^{61}+4 q^{62}+5 q^{63}+2 q^{64}+2 q^{65}\\
&\quad\quad\quad\quad+q^{66})\\
&\quad
-x^{11}(q^{53}+q^{54}+3 q^{55}+3 q^{56}+4 q^{57}+5 q^{58}+7 q^{59}+6 q^{60}+9 q^{61}+9 q^{62}+7 q^{63}\\
&\quad\quad\quad\quad+10 q^{64}+7 q^{65}+6 q^{66}+7 q^{67}+4 q^{68}+3 q^{69}+3 q^{70}+q^{72})\\
&\quad
-x^{12}(q^{60}+q^{62}+2 q^{63}+q^{64}+3 q^{65}+3 q^{66}+2 q^{67}+4 q^{68}+3 q^{69}+2 q^{70}+4 q^{71}\\
&\quad\quad\quad\quad+q^{72}+2 q^{73}+2 q^{74}+q^{76}),\\
\intertext{and}
p_{12}(x,q)&=x^{12}q^{90} + x^{13}(q^{91}+q^{92}+2 q^{94}+q^{96}+q^{97})\\
&\quad+ x^{14}(q^{93}+2 q^{95}+q^{96}+q^{97}+2 q^{98}+q^{99}+q^{100}+2 q^{101}+q^{103})\\
&\quad+x^{15}(q^{97}+q^{99}+q^{100}+2 q^{102}+q^{104}+q^{105}+q^{107}).
\end{align*}}
\end{theorem}

\begin{theorem}
It holds that
\begin{align}
p_0(x,q)G_{\ps_{T_{\mathrm{III},2}}}(x)&+p_3(x,q)G_{\ps_{T_{\mathrm{III},2}}}(xq^3)+p_6(x,q)G_{\ps_{T_{\mathrm{III},2}}}(x q^6)\nonumber\\
&+p_9(x,q)G_{\ps_{T_{\mathrm{III},2}}}(xq^9)+p_{12}(x,q)G_{\ps_{T_{\mathrm{III},2}}}(xq^{12})=0,
\end{align}
where
{\footnotesize\begin{align*}
p_0(x,q)&=1+x(q^{5}+q^{6}+2 q^{8}+q^{10}+q^{11})\\
&\quad+ x^2(q^{11}+2 q^{13}+q^{14}+q^{15}+2 q^{16}+q^{17}+q^{18}+2 q^{19}+q^{21})\\
&\quad+ x^3(q^{19}+q^{21}+q^{22}+2 q^{24}+q^{26}+q^{27}+q^{29}),\\
p_3(x,q)&=-1-x(q^{2}+q^{3}+q^{4}+q^{5}+q^{6}+2 q^{8}+q^{10}+q^{11})\\
&\quad
-x^2(q^{4}+q^{5}+2 q^{6}+2 q^{7}+3 q^{8}+2 q^{9}+3 q^{10}+3 q^{11}+3 q^{12}+4 q^{13}+3 q^{14}+2 q^{15}\\
&\quad\quad\quad\quad+2 q^{16}+q^{17}+q^{18}+2 q^{19}+q^{21})\\
&\quad
-  x^3(q^{7}+q^{8}+2 q^{9}+4 q^{10}+3 q^{11}+5 q^{12}+5 q^{13}+6 q^{14}+7 q^{15}+8 q^{16}+7 q^{17}+6 q^{18}\\
&\quad\quad\quad\quad+6 q^{19}+4 q^{20}+5 q^{21}+4 q^{22}+3 q^{23}+3 q^{24}+q^{25}+q^{26}+q^{27}+q^{29})\\
&\quad
-x^4(q^{12}+2 q^{13}+2 q^{14}+5 q^{15}+4 q^{16}+6 q^{17}+10 q^{18}+8 q^{19}+10 q^{20}+11 q^{21}+8 q^{22}\\
&\quad\quad\quad\quad+10 q^{23}+9 q^{24}+8 q^{25}+7 q^{26}+6 q^{27}+4 q^{28}+3 q^{29}+2 q^{30}+2 q^{31}+q^{32}+q^{33})\\
&\quad
-x^5(q^{18}+3 q^{20}+3 q^{21}+4 q^{22}+7 q^{23}+6 q^{24}+7 q^{25}+10 q^{26}+7 q^{27}+9 q^{28}+9 q^{29}\\
&\quad\quad\quad\quad+6 q^{30}+7 q^{31}+5 q^{32}+4 q^{33}+3 q^{34}+3 q^{35}+q^{36}+q^{37})\\
&\quad
- x^6(q^{26}+2 q^{28}+2 q^{29}+q^{30}+4 q^{31}+2 q^{32}+3 q^{33}+4 q^{34}+2 q^{35}+3 q^{36}+3 q^{37}\\
&\quad\quad\quad\quad+q^{38}+2 q^{39}+q^{40}+q^{42}),\\
p_6(x,q)&=x^4(q^{16}+q^{18}+2 q^{20}+q^{22}+q^{24})\\
&\quad
+x^5(q^{18}+2 q^{20}+q^{21}+4 q^{22}+3 q^{23}+4 q^{24}+4 q^{25}+5 q^{26}+5 q^{27}+4 q^{28}+4 q^{29}\\
&\quad\quad\quad\quad+3 q^{30}+4 q^{31}+q^{32}+2 q^{33}+q^{35})\\
&\quad
+x^6(q^{23}+2 q^{24}+3 q^{25}+5 q^{26}+5 q^{27}+9 q^{28}+9 q^{29}+10 q^{30}+12 q^{31}+12 q^{32}\\
&\quad\quad\quad\quad+14 q^{33}+12 q^{34}+12 q^{35}+10 q^{36}+9 q^{37}+9 q^{38}+5 q^{39}+5 q^{40}+3 q^{41}\\
&\quad\quad\quad\quad+2 q^{42}+q^{43})\\
&\quad
+x^7(q^{29}+3 q^{30}+4 q^{31}+6 q^{32}+7 q^{33}+12 q^{34}+12 q^{35}+16 q^{36}+18 q^{37}+16 q^{38}\\
&\quad\quad\quad\quad+19 q^{39}+19 q^{40}+16 q^{41}+18 q^{42}+16 q^{43}+12 q^{44}+12 q^{45}+7 q^{46}+6 q^{47}\\
&\quad\quad\quad\quad+4 q^{48}+3 q^{49}+q^{50})\\
&\quad
+x^8(q^{36}+2 q^{37}+3 q^{38}+6 q^{39}+6 q^{40}+9 q^{41}+11 q^{42}+11 q^{43}+13 q^{44}+16 q^{45}\\
&\quad\quad\quad\quad+12 q^{46}+16 q^{47}+13 q^{48}+11 q^{49}+11 q^{50}+9 q^{51}+6 q^{52}+6 q^{53}+3 q^{54}\\
&\quad\quad\quad\quad+2 q^{55}+q^{56})\\
&\quad
+x^9(q^{44}+q^{45}+q^{46}+3 q^{47}+2 q^{48}+3 q^{49}+5 q^{50}+3 q^{51}+5 q^{52}+5 q^{53}+3 q^{54}\\
&\quad\quad\quad\quad+5 q^{55}+3 q^{56}+2 q^{57}+3 q^{58}+q^{59}+q^{60}+q^{61}),\\
p_{9}(x,q)&=-x^6q^{36}
-x^7(q^{38}+q^{39}+2 q^{41}+q^{43}+q^{44}+q^{45}+q^{46}+q^{47})\\
&\quad
-x^8(q^{41}+2 q^{43}+q^{44}+q^{45}+2 q^{46}+2 q^{47}+3 q^{48}+4 q^{49}+3 q^{50}+3 q^{51}+3 q^{52}\\
&\quad\quad\quad\quad+2 q^{53}+3 q^{54}+2 q^{55}+2 q^{56}+q^{57}+q^{58})\\
&\quad
-x^9(q^{46}+q^{48}+q^{49}+q^{50}+3 q^{51}+3 q^{52}+4 q^{53}+5 q^{54}+4 q^{55}+6 q^{56}+6 q^{57}\\
&\quad\quad\quad\quad+7 q^{58}+8 q^{59}+7 q^{60}+6 q^{61}+5 q^{62}+5 q^{63}+3 q^{64}+4 q^{65}+2 q^{66}+q^{67}\\
&\quad\quad\quad\quad+q^{68})\\
&\quad
-x^{10}(q^{55}+q^{56}+2 q^{57}+2 q^{58}+3 q^{59}+4 q^{60}+6 q^{61}+7 q^{62}+8 q^{63}+9 q^{64}+10 q^{65}\\
&\quad\quad\quad\quad+8 q^{66}+11 q^{67}+10 q^{68}+8 q^{69}+10 q^{70}+6 q^{71}+4 q^{72}+5 q^{73}+2 q^{74}+2 q^{75}\\
&\quad\quad\quad\quad+q^{76})\\
&\quad
-x^{11}(q^{64}+q^{65}+3 q^{66}+3 q^{67}+4 q^{68}+5 q^{69}+7 q^{70}+6 q^{71}+9 q^{72}+9 q^{73}+7 q^{74}\\
&\quad\quad\quad\quad+10 q^{75}+7 q^{76}+6 q^{77}+7 q^{78}+4 q^{79}+3 q^{80}+3 q^{81}+q^{83})\\
&\quad
-x^{12}(q^{72}+q^{74}+2 q^{75}+q^{76}+3 q^{77}+3 q^{78}+2 q^{79}+4 q^{80}+3 q^{81}+2 q^{82}+4 q^{83}\\
&\quad\quad\quad\quad+q^{84}+2 q^{85}+2 q^{86}+q^{88}),\\
\intertext{and}
p_{12}(x,q)&=x^{12}q^{102} + x^{13}(q^{104}+q^{105}+2 q^{107}+q^{109}+q^{110})\\
&\quad+ x^{14}(q^{107}+2 q^{109}+q^{110}+q^{111}+2 q^{112}+q^{113}+q^{114}+2 q^{115}+q^{117})\\
&\quad+x^{15}(q^{112}+q^{114}+q^{115}+2 q^{117}+q^{119}+q^{120}+q^{122}).
\end{align*}}
\end{theorem}

\begin{theorem}
It holds that
\begin{align}
p_0(x,q)G_{\ps_{T_{\mathrm{III},a}}}(x)&+p_3(x,q)G_{\ps_{T_{\mathrm{III},a}}}(xq^3)+p_6(x,q)G_{\ps_{T_{\mathrm{III},a}}}(x q^6)\nonumber\\
&+p_9(x,q)G_{\ps_{T_{\mathrm{III},a}}}(xq^9)+p_{12}(x,q)G_{\ps_{T_{\mathrm{III},a}}}(xq^{12})=0,
\end{align}
where
{\footnotesize\begin{align*}
p_0(x,q)&=1+x(q^4+q^5+q^7+q^8+q^{10}+q^{11})\\
&\quad+ x^2(q^9+q^{11}+2 q^{12}+q^{14}+2 q^{15}+q^{16}+2 q^{18}+q^{19}+q^{21})\\
&\quad+ x^3(q^{16}+q^{19}+q^{20}+q^{22}+q^{23}+q^{25}+q^{26}+q^{29}),\\
p_3(x,q)&=-1-x(q+q^2+q^3+q^4+q^5+q^7+q^8+q^{10}+q^{11})\\
&\quad
-x^2(q^{3}+2 q^{4}+2 q^{5}+3 q^{6}+2 q^{7}+3 q^{8}+3 q^{9}+2 q^{10}+3 q^{11}+4 q^{12}+2 q^{13}+2 q^{14}\\
&\quad\quad\quad\quad+2 q^{15}+q^{16}+2 q^{18}+q^{19}+q^{21})\\
&\quad
-  x^3(q^{6}+3 q^{7}+3 q^{8}+4 q^{9}+5 q^{10}+5 q^{11}+5 q^{12}+7 q^{13}+7 q^{14}+7 q^{15}+7 q^{16}+5 q^{17}\\
&\quad\quad\quad\quad+4 q^{18}+5 q^{19}+4 q^{20}+3 q^{21}+3 q^{22}+2 q^{23}+q^{24}+q^{25}+q^{26}+q^{29})\\
&\quad
-x^4(q^{10}+3 q^{11}+3 q^{12}+3 q^{13}+6 q^{14}+7 q^{15}+7 q^{16}+10 q^{17}+10 q^{18}+8 q^{19}+9 q^{20}\\
&\quad\quad\quad\quad+9 q^{21}+8 q^{22}+8 q^{23}+7 q^{24}+5 q^{25}+5 q^{26}+4 q^{27}+2 q^{28}+2 q^{29}+q^{30}\\
&\quad\quad\quad\quad+q^{31}+q^{32})\\
&\quad
-x^5(q^{15}+q^{16}+q^{17}+4 q^{18}+5 q^{19}+3 q^{20}+6 q^{21}+8 q^{22}+5 q^{23}+7 q^{24}+10 q^{25}\\
&\quad\quad\quad\quad+6 q^{26}+6 q^{27}+9 q^{28}+5 q^{29}+4 q^{30}+5 q^{31}+3 q^{32}+2 q^{33}+3 q^{34}+q^{35}\\
&\quad\quad\quad\quad+q^{37})\\
&\quad
- x^6(q^{22}+q^{23}+q^{25}+3 q^{26}+q^{27}+q^{28}+4 q^{29}+2 q^{30}+q^{31}+4 q^{32}+3 q^{33}+3 q^{35}\\
&\quad\quad\quad\quad+3 q^{36}+q^{38}+2 q^{39}+q^{42}),\\
p_6(x,q)&=x^4(q^{16}+q^{18}+q^{19}+q^{20}+q^{21}+q^{23})\\
&\quad
+x^5(q^{17}+q^{19}+3 q^{20}+3 q^{21}+3 q^{22}+4 q^{23}+5 q^{24}+4 q^{25}+4 q^{26}+5 q^{27}+4 q^{28}\\
&\quad\quad\quad\quad+3 q^{29}+3 q^{30}+3 q^{31}+q^{32}+q^{34})\\
&\quad
+x^6(q^{21}+2 q^{22}+2 q^{23}+4 q^{24}+6 q^{25}+8 q^{26}+8 q^{27}+10 q^{28}+10 q^{29}+11 q^{30}\\
&\quad\quad\quad\quad+13 q^{31}+13 q^{32}+11 q^{33}+10 q^{34}+10 q^{35}+8 q^{36}+8 q^{37}+6 q^{38}\\
&\quad\quad\quad\quad+4 q^{39}+2 q^{40}+2 q^{41}+q^{42})\\
&\quad
+x^7(q^{26}+3 q^{27}+3 q^{28}+4 q^{29}+7 q^{30}+8 q^{31}+11 q^{32}+14 q^{33}+14 q^{34}+15 q^{35}\\
&\quad\quad\quad\quad+17 q^{36}+17 q^{37}+17 q^{38}+17 q^{39}+15 q^{40}+14 q^{41}+14 q^{42}+11 q^{43}\\
&\quad\quad\quad\quad+8 q^{44}+7 q^{45}+4 q^{46}+3 q^{47}+3 q^{48}+q^{49})\\
&\quad
+x^8(q^{31}+q^{33}+4 q^{34}+3 q^{35}+4 q^{36}+8 q^{37}+8 q^{38}+7 q^{39}+12 q^{40}+12 q^{41}+10 q^{42}\\
&\quad\quad\quad\quad+14 q^{43}+14 q^{44}+10 q^{45}+12 q^{46}+12 q^{47}+7 q^{48}+8 q^{49}+8 q^{50}+4 q^{51}\\
&\quad\quad\quad\quad+3 q^{52}+4 q^{53}+q^{54}+q^{56})\\
&\quad
+x^9(q^{38}+2 q^{41}+2 q^{42}+3 q^{44}+4 q^{45}+q^{46}+3 q^{47}+6 q^{48}+2 q^{49}+2 q^{50}+6 q^{51}\\
&\quad\quad\quad\quad+3 q^{52}+q^{53}+4 q^{54}+3 q^{55}+2 q^{57}+2 q^{58}+q^{61}),\\
p_{9}(x,q)&=-x^6q^{36}
-x^7(q^{37}+q^{38}+q^{40}+q^{41}+q^{43}+q^{44}+q^{45}+q^{46}+q^{47})\\
&\quad
-x^8(q^{39}+q^{41}+2 q^{42}+q^{44}+2 q^{45}+2 q^{46}+2 q^{47}+4 q^{48}+3 q^{49}+2 q^{50}+3 q^{51}\\
&\quad\quad\quad\quad+3 q^{52}+2 q^{53}+3 q^{54}+2 q^{55}+2 q^{56}+q^{57})\\
&\quad
-x^9(q^{43}+q^{46}+q^{47}+q^{48}+2 q^{49}+3 q^{50}+3 q^{51}+4 q^{52}+5 q^{53}+4 q^{54}+5 q^{55}\\
&\quad\quad\quad\quad+7 q^{56}+7 q^{57}+7 q^{58}+7 q^{59}+5 q^{60}+5 q^{61}+5 q^{62}+4 q^{63}+3 q^{64}\\
&\quad\quad\quad\quad+3 q^{65}+q^{66})\\
&\quad
-x^{10}(q^{52}+q^{53}+q^{54}+2 q^{55}+2 q^{56}+4 q^{57}+5 q^{58}+5 q^{59}+7 q^{60}+8 q^{61}+8 q^{62}\\
&\quad\quad\quad\quad+9 q^{63}+9 q^{64}+8 q^{65}+10 q^{66}+10 q^{67}+7 q^{68}+7 q^{69}+6 q^{70}+3 q^{71}\\
&\quad\quad\quad\quad+3 q^{72}+3 q^{73}+q^{74})\\
&\quad
-x^{11}(q^{59}+q^{61}+3 q^{62}+2 q^{63}+3 q^{64}+5 q^{65}+4 q^{66}+5 q^{67}+9 q^{68}+6 q^{69}+6 q^{70}\\
&\quad\quad\quad\quad+10 q^{71}+7 q^{72}+5 q^{73}+8 q^{74}+6 q^{75}+3 q^{76}+5 q^{77}+4 q^{78}+q^{79}+q^{80}\\
&\quad\quad\quad\quad+q^{81})\\
&\quad
-x^{12}(q^{66}+2 q^{69}+q^{70}+3 q^{72}+3 q^{73}+3 q^{75}+4 q^{76}+q^{77}+2 q^{78}+4 q^{79}+q^{80}\\
&\quad\quad\quad\quad+q^{81}+3 q^{82}+q^{83}+q^{85}+q^{86}),\\
\intertext{and}
p_{12}(x,q)&=x^{12}q^{99} + x^{13}(q^{100}+q^{101}+q^{103}+q^{104}+q^{106}+q^{107})\\
&\quad+ x^{14}(q^{102}+q^{104}+2 q^{105}+q^{107}+2 q^{108}+q^{109}+2 q^{111}+q^{112}+q^{114})\\
&\quad+x^{15}(q^{106}+q^{109}+q^{110}+q^{112}+q^{113}+q^{115}+q^{116}+q^{119}).
\end{align*}}
\end{theorem}


\subsection{Partition set of type IV}

Recall that the partition set of type IV is the set of partitions with difference at least $3$ at distance $3$ such that if parts at distance $2$ differ by at most $1$, then the sum of the two parts and their intermediate part is congruent to $2$ modulo $3$. In other words, if $\lambda=\lambda_1+\lambda_2+\cdots+\lambda_\ell$ is in this partition set, then
\begin{enumerate}[label=(\roman*)]
\item $\lambda_i-\lambda_{i+3}\ge 3$;
\item $\lambda_i-\lambda_{i+2}\le 1$ implies $\lambda_i+\lambda_{i+1}+\lambda_{i+2}\equiv 2 \pmod{3}$.
\end{enumerate}

Let $\ps_{T_{\mathrm{IV}}}$ denote the partition set of type IV.

\begin{claim}
$\ps_{T_{\mathrm{IV}}}$ is a partition ideal of modulus $3$.
\end{claim}

\begin{claim}
The $L_{\ps_{T_{\mathrm{IV}}}}$ corresponding to modulus $3$ equals
\begin{align*}
\{&\emptyset,\;1,\;1+1,\;2,\;2+1,\;2+2,\;2+2+1,\;3,\;3+1,\\
&3+1+1,\;3+2,\;3+2+1,\;3+3,\;3+3+1,\;3+3+2\}.
\end{align*}
\end{claim}

\begin{claim}\label{claim:T4-3}
The span and linking set of partitions in $L_{\ps_{T_{\mathrm{IV}}}}$ are given as follows.
{\footnotesize\begin{equation*}
\begin{array}{lcp{0.05cm}l}
& \text{span} && \quad\quad\quad\quad\quad\quad\quad\quad\quad\quad\quad\quad\text{linking set}\\
\pi_0 = \emptyset & 1 && \{\pi_0,\; \pi_1,\; \pi_2,\; \pi_3,\; \pi_4,\; \pi_5,\; \pi_6,\; \pi_7,\; \pi_8,\; \pi_9,\; \pi_{10},\; \pi_{11},\; \pi_{12},\; \pi_{13},\; \pi_{14}\}\\
\pi_1 = 1 & 1 && \{\pi_0,\; \pi_1,\; \pi_2,\; \pi_3,\; \pi_4,\; \pi_5,\; \pi_6,\; \pi_7,\; \pi_8,\; \pi_9,\; \pi_{10},\; \pi_{11},\; \pi_{12},\; \pi_{13},\; \pi_{14}\}\\
\pi_2 = 1+1 & 1 && \{\pi_0,\; \pi_1,\; \pi_2,\; \pi_3,\; \pi_4,\; \pi_5,\; \pi_6,\; \pi_7,\; \pi_8,\; \pi_9,\; \pi_{10},\; \pi_{11},\; \pi_{12},\; \pi_{13},\; \pi_{14}\}\\
\pi_3 = 2 & 1 && \{\pi_0,\; \pi_1,\; \pi_2,\; \pi_3,\; \pi_4,\; \pi_5,\; \pi_6,\; \pi_7,\; \pi_8,\; \pi_9,\; \pi_{10},\; \pi_{11},\; \pi_{12},\; \pi_{13},\; \pi_{14}\}\\
\pi_4 = 2+1 & 1 && \{\pi_0,\; \pi_1,\; \pi_2,\; \pi_3,\; \pi_4,\; \pi_5,\; \pi_6,\; \pi_7,\; \pi_8,\; \pi_9,\; \pi_{10},\; \pi_{11},\; \pi_{12},\; \pi_{13},\; \pi_{14}\}\\
\pi_5 = 2+2 & 1 && \{\pi_0,\; \pi_1,\; \pi_3,\; \pi_4,\; \pi_5,\; \pi_6,\; \pi_7,\; \pi_8,\; \pi_{10},\; \pi_{11},\; \pi_{12},\; \pi_{13},\; \pi_{14}\}\\
\pi_6 = 2+2+1 & 1 && \{\pi_0,\; \pi_1,\; \pi_3,\; \pi_4,\; \pi_5,\; \pi_6,\; \pi_7,\; \pi_8,\; \pi_{10},\; \pi_{11},\; \pi_{12},\; \pi_{13},\; \pi_{14}\}\\
\pi_7 = 3 & 1 && \{\pi_0,\; \pi_1,\; \pi_2,\; \pi_3,\; \pi_4,\; \pi_5,\; \pi_7,\; \pi_8,\; \pi_9,\; \pi_{10},\; \pi_{11},\; \pi_{12},\; \pi_{13},\; \pi_{14}\}\\
\pi_8 = 3+1 & 1 && \{\pi_0,\; \pi_1,\; \pi_2,\; \pi_3,\; \pi_4,\; \pi_5,\; \pi_7,\; \pi_8,\; \pi_9,\; \pi_{10},\; \pi_{11},\; \pi_{12},\; \pi_{13},\; \pi_{14}\}\\
\pi_9 = 3+1+1 & 1 && \{\pi_0,\; \pi_1,\; \pi_2,\; \pi_3,\; \pi_4,\; \pi_5,\; \pi_7,\; \pi_8,\; \pi_9,\; \pi_{10},\; \pi_{11},\; \pi_{12},\; \pi_{13},\; \pi_{14}\}\\
\pi_{10} = 3+2 & 1 && \{\pi_0,\; \pi_1,\; \pi_3,\; \pi_4,\; \pi_5,\; \pi_7,\; \pi_8,\; \pi_{10},\; \pi_{11},\; \pi_{12},\; \pi_{13},\; \pi_{14}\}\\
\pi_{11} = 3+2+1 & 1 && \{\pi_0,\; \pi_1,\; \pi_3,\; \pi_4,\; \pi_5,\; \pi_7,\; \pi_8,\; \pi_{10},\; \pi_{11},\; \pi_{12},\; \pi_{13},\; \pi_{14}\}\\
\pi_{12} = 3+3 & 1 && \{\pi_0,\; \pi_3,\; \pi_7,\; \pi_{10},\; \pi_{12},\; \pi_{14}\}\\
\pi_{13} = 3+3+1 & 1 && \{\pi_0,\; \pi_3,\; \pi_7,\; \pi_{10},\; \pi_{12},\; \pi_{14}\}\\
\pi_{14} = 3+3+2 & 1 && \{\pi_0,\; \pi_3,\; \pi_7,\; \pi_{10},\; \pi_{12},\; \pi_{14}\}\\
\end{array}
\end{equation*}}
\end{claim}

Let us denote by $H_i(x)=H_i(x,q)$ the generating function of partitions $\lambda$ in $\ps_{T_{\mathrm{IV}}}$ with $\tail_m(\lambda)=\pi_i$ for $i=0,1,\ldots, 14$ where the $\pi_i$'s are as defined in Claim \ref{claim:T4-3}.

Following \eqref{eq:important-gf-2}, we have
{\footnotesize\begin{align}
H_0(x)&=H_0(xq^3)+H_1(xq^3)+H_2(xq^3)+H_3(xq^3)+H_4(xq^3)+H_5(xq^3)\notag\\
&\quad+H_6(xq^3)+H_7(xq^3)+H_8(xq^3)+H_9(xq^3)+H_{10}(xq^3)\notag\\
&\quad+H_{11}(xq^3)+H_{12}(xq^3)+H_{13}(xq^3)+H_{14}(xq^3),\\
x^{-1}q^{-1}H_1(x)&=H_0(xq^3)+H_1(xq^3)+H_2(xq^3)+H_3(xq^3)+H_4(xq^3)+H_5(xq^3)\notag\\
&\quad+H_6(xq^3)+H_7(xq^3)+H_8(xq^3)+H_9(xq^3)+H_{10}(xq^3)\notag\\
&\quad+H_{11}(xq^3)+H_{12}(xq^3)+H_{13}(xq^3)+H_{14}(xq^3),\\
x^{-2}q^{-2}H_2(x)&=H_0(xq^3)+H_1(xq^3)+H_2(xq^3)+H_3(xq^3)+H_4(xq^3)+H_5(xq^3)\notag\\
&\quad+H_6(xq^3)+H_7(xq^3)+H_8(xq^3)+H_9(xq^3)+H_{10}(xq^3)\notag\\
&\quad+H_{11}(xq^3)+H_{12}(xq^3)+H_{13}(xq^3)+H_{14}(xq^3),\\
x^{-1}q^{-2}H_3(x)&=H_0(xq^3)+H_1(xq^3)+H_2(xq^3)+H_3(xq^3)+H_4(xq^3)+H_5(xq^3)\notag\\
&\quad+H_6(xq^3)+H_7(xq^3)+H_8(xq^3)+H_9(xq^3)+H_{10}(xq^3)\notag\\
&\quad+H_{11}(xq^3)+H_{12}(xq^3)+H_{13}(xq^3)+H_{14}(xq^3),\\
x^{-2}q^{-3}H_4(x)&=H_0(xq^3)+H_1(xq^3)+H_2(xq^3)+H_3(xq^3)+H_4(xq^3)+H_5(xq^3)\notag\\
&\quad+H_6(xq^3)+H_7(xq^3)+H_8(xq^3)+H_9(xq^3)+H_{10}(xq^3)\notag\\
&\quad+H_{11}(xq^3)+H_{12}(xq^3)+H_{13}(xq^3)+H_{14}(xq^3),\\
x^{-2}q^{-4}H_5(x)&=H_0(xq^3)+H_1(xq^3)+H_3(xq^3)+H_4(xq^3)+H_5(xq^3)+H_6(xq^3)\notag\\
&\quad+H_7(xq^3)+H_8(xq^3)+H_{10}(xq^3)+H_{11}(xq^3)+H_{12}(xq^3)\notag\\
&\quad+H_{13}(xq^3)+H_{14}(xq^3),\\
x^{-3}q^{-5}H_6(x)&=H_0(xq^3)+H_1(xq^3)+H_3(xq^3)+H_4(xq^3)+H_5(xq^3)+H_6(xq^3)\notag\\
&\quad+H_7(xq^3)+H_8(xq^3)+H_{10}(xq^3)+H_{11}(xq^3)+H_{12}(xq^3)\notag\\
&\quad+H_{13}(xq^3)+H_{14}(xq^3),\\
x^{-1}q^{-3}H_7(x)&=H_0(xq^3)+H_1(xq^3)+H_2(xq^3)+H_3(xq^3)+H_4(xq^3)+H_5(xq^3)\notag\\
&\quad+H_7(xq^3)+H_8(xq^3)+H_9(xq^3)+H_{10}(xq^3)+H_{11}(xq^3)\notag\\
&\quad+H_{12}(xq^3)+H_{13}(xq^3)+H_{14}(xq^3),\\
x^{-2}q^{-4}H_8(x)&=H_0(xq^3)+H_1(xq^3)+H_2(xq^3)+H_3(xq^3)+H_4(xq^3)+H_5(xq^3)\notag\\
&\quad+H_7(xq^3)+H_8(xq^3)+H_9(xq^3)+H_{10}(xq^3)+H_{11}(xq^3)\notag\\
&\quad+H_{12}(xq^3)+H_{13}(xq^3)+H_{14}(xq^3),\\
x^{-3}q^{-5}H_9(x)&=H_0(xq^3)+H_1(xq^3)+H_2(xq^3)+H_3(xq^3)+H_4(xq^3)+H_5(xq^3)\notag\\
&\quad+H_7(xq^3)+H_8(xq^3)+H_9(xq^3)+H_{10}(xq^3)+H_{11}(xq^3)\notag\\
&\quad+H_{12}(xq^3)+H_{13}(xq^3)+H_{14}(xq^3),\\
x^{-2}q^{-5}H_{10}(x)&=H_0(xq^3)+H_1(xq^3)+H_3(xq^3)+H_4(xq^3)+H_5(xq^3)+H_7(xq^3)\notag\\
&\quad+H_8(xq^3)+H_{10}(xq^3)+H_{11}(xq^3)+H_{12}(xq^3)+H_{13}(xq^3)\notag\\
&\quad+H_{14}(xq^3),\\
x^{-3}q^{-6}H_{11}(x)&=H_0(xq^3)+H_1(xq^3)+H_3(xq^3)+H_4(xq^3)+H_5(xq^3)+H_7(xq^3)\notag\\
&\quad+H_8(xq^3)+H_{10}(xq^3)+H_{11}(xq^3)+H_{12}(xq^3)+H_{13}(xq^3)\notag\\
&\quad+H_{14}(xq^3),\\
x^{-2}q^{-6}H_{12}(x)&=H_0(xq^3)+H_3(xq^3)+H_7(xq^3)+H_{10}(xq^3)+H_{12}(xq^3)+H_{14}(xq^3),\\
x^{-3}q^{-7}H_{13}(x)&=H_0(xq^3)+H_3(xq^3)+H_7(xq^3)+H_{10}(xq^3)+H_{12}(xq^3)+H_{14}(xq^3),\\
x^{-3}q^{-8}H_{14}(x)&=H_0(xq^3)+H_3(xq^3)+H_7(xq^3)+H_{10}(xq^3)+H_{12}(xq^3)+H_{14}(xq^3).
\end{align}}

This system may be simplified as
\begin{framed}
{\begin{align}
H_0(x)&=(1+xq^4+x^2q^8+xq^5+x^2q^9)H_0(xq^3)+(1+xq^4)H_5(xq^3)\notag\\
&\quad+(1+xq^4+x^2q^8)H_7(xq^3)+(1+xq^4)H_{10}(xq^3)\notag\\
&\quad+(1+xq^4+xq^5)H_{12}(xq^3),\\
H_5(x)&=(x^2q^4+x^3q^8+x^3q^9+x^4q^{13})H_0(xq^3)+(x^2q^4+x^3q^8)H_5(xq^3)\notag\\
&\quad+(x^2q^4+x^3q^8)H_7(xq^3)+(x^2q^4+x^3q^8)H_{10}(xq^3)\notag\\
&\quad+(x^2q^4+x^3q^8+x^3q^9)H_{12}(xq^3),\\
H_7(x)&=(xq^3+x^2q^7+x^3q^{11}+x^2q^8+x^3q^{12})H_0(xq^3)+xq^3H_5(xq^3)\notag\\
&\quad+(xq^3+x^2q^7+x^3q^{11})H_7(xq^3)+(xq^3+x^2q^7)H_{10}(xq^3)\notag\\
&\quad+(xq^3+x^2q^7+x^2q^8)H_{12}(xq^3),\\
H_{10}(x)&=(x^2q^5+x^3q^9+x^3q^{10}+x^4q^{14})H_0(xq^3)+x^2q^5H_5(xq^3)\notag\\
&\quad+(x^2q^5+x^3q^9)H_7(xq^3)+(x^2q^5+x^3q^9)H_{10}(xq^3)\notag\\
&\quad+(x^2q^5+x^3q^9+x^3q^{10})H_{12}(xq^3),\\
H_{12}(x)&=(x^2q^6+x^3q^{11})H_0(xq^3)+x^2q^6H_7(xq^3)+x^2q^6H_{10}(xq^3)\notag\\
&\quad+(x^2q^6+x^3q^{11})H_{12}(xq^3).
\end{align}}
\end{framed}

Let $G_{\ps_{T_{\mathrm{IV},1}}}(x)=G_{\ps_{T_{\mathrm{II},1}}}(x,q)$ denote the generating function of partitions in $\ps_{T_{\mathrm{IV}}}$ whose smallest part is at least $1$.

Let $G_{\ps_{T_{\mathrm{IV},a}}}(x)$ denote the generating function of partitions in $\ps_{T_{\mathrm{IV}}}$ where $1$ appears at most once.

Let $G_{\ps_{T_{\mathrm{IV},b}}}(x)$ denote the generating function of partitions in $\ps_{T_{\mathrm{IV}}}$ where the smallest part is at least $2$ with $2$ appearing at most once.

It follows that
\begin{align}
G_{\ps_{T_{\mathrm{IV},1}}}(x)&=H_0(x)+H_1(x)+H_2(x)+H_3(x)+H_4(x)+H_5(x)\notag\\
&\quad+H_6(x)+H_7(x)+H_8(x)+H_9(x)+H_{10}(x)\notag\\
&\quad+H_{11}(x)+H_{12}(x)+H_{13}(x)+H_{14}(x)\notag\\
&=H_0(xq^{-3}),\\
G_{\ps_{T_{\mathrm{IV},a}}}(x)&=H_0(x)+H_1(x)+H_3(x)+H_4(x)+H_5(x)+H_6(x)\notag\\
&\quad+H_7(x)+H_8(x)+H_{10}(x)+H_{11}(x)+H_{12}(x)\notag\\
&\quad+H_{13}(x)+H_{14}(x)\notag\\
&=x^{-2}q^2H_5(xq^{-3}),\\
G_{\ps_{T_{\mathrm{IV},b}}}(x)&=H_0(x)+H_3(x)+H_7(x)+H_{10}(x)+H_{12}(x)+H_{14}(x)\notag\\
&=x^{-2}H_{12}(xq^{-3}).
\end{align}

Analogously, we can use the algorithm in Section \ref{sec:q-diff} to deduce the following $q$-difference equations for $G_{\ps_{T_{\mathrm{IV},1}}}(x)$, $G_{\ps_{T_{\mathrm{IV},a}}}(x)$ and $G_{\ps_{T_{\mathrm{IV},b}}}(x)$, respectively.

\begin{theorem}
It holds that
\begin{align}
p_0(x,q)G_{\ps_{T_{\mathrm{IV},1}}}(x)&+p_3(x,q)G_{\ps_{T_{\mathrm{IV},1}}}(xq^3)+p_6(x,q)G_{\ps_{T_{\mathrm{IV},1}}}(x q^6)\nonumber\\
&+p_9(x,q)G_{\ps_{T_{\mathrm{IV},1}}}(xq^9)+p_{12}(x,q)G_{\ps_{T_{\mathrm{IV},1}}}(xq^{12})=0,
\end{align}
where
{\footnotesize\begin{align*}
p_0(x,q)&=1+x(q^{4}+q^{5}+q^{7}+q^{8}+q^{10}+q^{11})\\
&\quad+ x^2(q^{9}+q^{11}+2 q^{12}+q^{14}+2 q^{15}+q^{16}+2 q^{18}+q^{19}+q^{21})\\
&\quad+ x^3(q^{16}+q^{19}+q^{20}+q^{22}+q^{23}+q^{25}+q^{26}+q^{29}),\\
p_3(x,q)&=-1-x(q+q^{2}+q^{3}+q^{4}+q^{5}+q^{7}+q^{8}+q^{10}+q^{11})\\
&\quad
-x^2(q^{2}+q^{3}+2 q^{4}+2 q^{5}+3 q^{6}+2 q^{7}+2 q^{8}+3 q^{9}+2 q^{10}+3 q^{11}+4 q^{12}+2 q^{13}\\
&\quad\quad\quad\quad+2 q^{14}+2 q^{15}+q^{16}+2 q^{18}+q^{19}+q^{21})\\
&\quad
-  x^3(2 q^{5}+2 q^{6}+3 q^{7}+4 q^{8}+4 q^{9}+4 q^{10}+5 q^{11}+6 q^{12}+7 q^{13}+7 q^{14}+6 q^{15}\\
&\quad\quad\quad\quad+6 q^{16}+5 q^{17}+3 q^{18}+4 q^{19}+4 q^{20}+3 q^{21}+3 q^{22}+2 q^{23}+q^{24}+q^{25}\\
&\quad\quad\quad\quad+q^{26}+q^{29})\\
&\quad
-x^4(2 q^{9}+3 q^{10}+2 q^{11}+4 q^{12}+6 q^{13}+6 q^{14}+8 q^{15}+10 q^{16}+8 q^{17}+9 q^{18}+9 q^{19}\\
&\quad\quad\quad\quad+7 q^{20}+8 q^{21}+8 q^{22}+7 q^{23}+6 q^{24}+5 q^{25}+3 q^{26}+3 q^{27}+2 q^{28}+q^{29}\\
&\quad\quad\quad\quad+q^{30}+q^{31}+q^{32})\\
&\quad
-x^5(q^{14}+q^{15}+2 q^{16}+5 q^{17}+3 q^{18}+4 q^{19}+8 q^{20}+6 q^{21}+5 q^{22}+9 q^{23}+8 q^{24}\\
&\quad\quad\quad\quad+6 q^{25}+9 q^{26}+6 q^{27}+4 q^{28}+6 q^{29}+4 q^{30}+2 q^{31}+3 q^{32}+2 q^{33}+q^{34}\\
&\quad\quad\quad\quad+q^{35})\\
&\quad
- x^6(q^{21}+q^{22}+2 q^{24}+2 q^{25}+q^{26}+3 q^{27}+3 q^{28}+q^{29}+3 q^{30}+4 q^{31}+q^{32}+2 q^{33}\\
&\quad\quad\quad\quad+3 q^{34}+q^{35}+q^{36}+2 q^{37}+q^{40}),\\
p_6(x,q)&=x^4(q^{12}+q^{14}+q^{16}+q^{17}+q^{19}+q^{21})\\
&\quad
+x^5(q^{13}+q^{15}+2 q^{16}+2 q^{17}+3 q^{18}+3 q^{19}+4 q^{20}+4 q^{21}+4 q^{22}+4 q^{23}+4 q^{24}\\
&\quad\quad\quad\quad+4 q^{25}+3 q^{26}+3 q^{27}+2 q^{28}+2 q^{29}+q^{30}+q^{32})\\
&\quad
+x^6(q^{17}+q^{18}+2 q^{19}+3 q^{20}+4 q^{21}+7 q^{22}+7 q^{23}+8 q^{24}+8 q^{25}+11 q^{26}+11 q^{27}\\
&\quad\quad\quad\quad+12 q^{28}+12 q^{29}+11 q^{30}+11 q^{31}+8 q^{32}+8 q^{33}+7 q^{34}+7 q^{35}+4 q^{36}\\
&\quad\quad\quad\quad+3 q^{37}+2 q^{38}+q^{39}+q^{40})\\
&\quad
+x^7(2 q^{23}+3 q^{24}+3 q^{25}+4 q^{26}+7 q^{27}+9 q^{28}+10 q^{29}+14 q^{30}+14 q^{31}+16 q^{32}\\
&\quad\quad\quad\quad+17 q^{33}+15 q^{34}+15 q^{35}+17 q^{36}+16 q^{37}+14 q^{38}+14 q^{39}+10 q^{40}+9 q^{41}\\
&\quad\quad\quad\quad+7 q^{42}+4 q^{43}+3 q^{44}+3 q^{45}+2 q^{46})\\
&\quad
+x^8(q^{29}+q^{30}+3 q^{31}+4 q^{32}+4 q^{33}+8 q^{34}+8 q^{35}+7 q^{36}+11 q^{37}+13 q^{38}+10 q^{39}\\
&\quad\quad\quad\quad+14 q^{40}+14 q^{41}+10 q^{42}+13 q^{43}+11 q^{44}+7 q^{45}+8 q^{46}+8 q^{47}+4 q^{48}\\
&\quad\quad\quad\quad+4 q^{49}+3 q^{50}+q^{51}+q^{52})\\
&\quad
+x^9(q^{36}+q^{38}+2 q^{39}+q^{40}+2 q^{41}+4 q^{42}+2 q^{43}+3 q^{44}+5 q^{45}+3 q^{46}+3 q^{47}\\
&\quad\quad\quad\quad+5 q^{48}+3 q^{49}+2 q^{50}+4 q^{51}+2 q^{52}+q^{53}+2 q^{54}+q^{55}+q^{57}),\\
p_{9}(x,q)&=-x^6q^{30}
-x^7(q^{31}+q^{32}+q^{34}+q^{35}+q^{37}+q^{38}+q^{39}+q^{40}+q^{41})\\
&\quad
-x^8(q^{33}+q^{35}+2 q^{36}+q^{38}+2 q^{39}+2 q^{40}+2 q^{41}+4 q^{42}+3 q^{43}+2 q^{44}+3 q^{45}\\
&\quad\quad\quad\quad+2 q^{46}+2 q^{47}+3 q^{48}+2 q^{49}+2 q^{50}+q^{51}+q^{52})\\
&\quad
-x^9(q^{37}+q^{40}+q^{41}+q^{42}+2 q^{43}+3 q^{44}+3 q^{45}+4 q^{46}+4 q^{47}+3 q^{48}+5 q^{49}\\
&\quad\quad\quad\quad+6 q^{50}+6 q^{51}+7 q^{52}+7 q^{53}+6 q^{54}+5 q^{55}+4 q^{56}+4 q^{57}+4 q^{58}+3 q^{59}\\
&\quad\quad\quad\quad+2 q^{60}+2 q^{61})\\
&\quad
-x^{10}(q^{46}+q^{47}+q^{48}+q^{49}+2 q^{50}+3 q^{51}+3 q^{52}+5 q^{53}+6 q^{54}+7 q^{55}+8 q^{56}\\
&\quad\quad\quad\quad+8 q^{57}+7 q^{58}+9 q^{59}+9 q^{60}+8 q^{61}+10 q^{62}+8 q^{63}+6 q^{64}+6 q^{65}+4 q^{66}\\
&\quad\quad\quad\quad+2 q^{67}+3 q^{68}+2 q^{69})\\
&\quad
-x^{11}(q^{55}+q^{56}+2 q^{57}+3 q^{58}+2 q^{59}+4 q^{60}+6 q^{61}+4 q^{62}+6 q^{63}+9 q^{64}+6 q^{65}\\
&\quad\quad\quad\quad+8 q^{66}+9 q^{67}+5 q^{68}+6 q^{69}+8 q^{70}+4 q^{71}+3 q^{72}+5 q^{73}+2 q^{74}+q^{75}+q^{76})\\
&\quad
-x^{12}(q^{62}+2 q^{65}+q^{66}+q^{67}+3 q^{68}+2 q^{69}+q^{70}+4 q^{71}+3 q^{72}+q^{73}+3 q^{74}+3 q^{75}\\
&\quad\quad\quad\quad+q^{76}+2 q^{77}+2 q^{78}+q^{80}+q^{81}),\\
\intertext{and}
p_{12}(x,q)&=x^{12}q^{93} + x^{13}(q^{94}+q^{95}+q^{97}+q^{98}+q^{100}+q^{101})\\
&\quad+ x^{14}(q^{96}+q^{98}+2 q^{99}+q^{101}+2 q^{102}+q^{103}+2 q^{105}+q^{106}+q^{108})\\
&\quad+x^{15}(q^{100}+q^{103}+q^{104}+q^{106}+q^{107}+q^{109}+q^{110}+q^{113}).
\end{align*}}
\end{theorem}

\begin{theorem}
It holds that
\begin{align}
p_0(x,q)G_{\ps_{T_{\mathrm{IV},a}}}(x)&+p_3(x,q)G_{\ps_{T_{\mathrm{IV},a}}}(xq^3)+p_6(x,q)G_{\ps_{T_{\mathrm{IV},a}}}(x q^6)\nonumber\\
&+p_9(x,q)G_{\ps_{T_{\mathrm{IV},a}}}(xq^9)+p_{12}(x,q)G_{\ps_{T_{\mathrm{IV},a}}}(xq^{12})=0,
\end{align}
where
{\footnotesize\begin{align*}
p_0(x,q)&=1+x(q^{4}+q^{5}+q^{6}+q^{7}+q^{8}+q^{9})\\
&\quad+ x^2(q^{9}+q^{10}+q^{11}+2 q^{12}+2 q^{13}+2 q^{14}+q^{15}+q^{16}+q^{17})\\
&\quad+ x^3(q^{16}+q^{17}+q^{18}+q^{19}+q^{20}+q^{21}+q^{22}+q^{23}),\\
p_3(x,q)&=-1-x(q+q^{2}+q^{3}+q^{4}+q^{5}+q^{6}+q^{7}+q^{8}+q^{9})\\
&\quad
-x^2(q^{3}+2 q^{4}+2 q^{5}+3 q^{6}+3 q^{7}+4 q^{8}+4 q^{9}+4 q^{10}+3 q^{11}+3 q^{12}+2 q^{13}+2 q^{14}\\
&\quad\quad\quad\quad+q^{15}+q^{16}+q^{17})\\
&\quad
-  x^3(q^{5}+q^{6}+2 q^{7}+4 q^{8}+5 q^{9}+6 q^{10}+8 q^{11}+8 q^{12}+8 q^{13}+8 q^{14}+8 q^{15}+7 q^{16}\\
&\quad\quad\quad\quad+6 q^{17}+4 q^{18}+3 q^{19}+2 q^{20}+q^{21}+q^{22}+q^{23})\\
&\quad
-x^4(q^{9}+2 q^{10}+3 q^{11}+5 q^{12}+7 q^{13}+8 q^{14}+9 q^{15}+11 q^{16}+13 q^{17}+13 q^{18}\\
&\quad\quad\quad\quad+11 q^{19}+10 q^{20}+8 q^{21}+6 q^{22}+5 q^{23}+4 q^{24}+3 q^{25}+q^{26})\\
&\quad
-x^5(q^{14}+2 q^{15}+2 q^{16}+4 q^{17}+6 q^{18}+8 q^{19}+10 q^{20}+10 q^{21}+10 q^{22}+9 q^{23}\\
&\quad\quad\quad\quad+8 q^{24}+8 q^{25}+7 q^{26}+5 q^{27}+3 q^{28}+q^{29}+q^{30}+q^{31})\\
&\quad
- x^6(q^{21}+2 q^{22}+2 q^{23}+3 q^{24}+3 q^{25}+3 q^{26}+4 q^{27}+4 q^{28}+3 q^{29}+2 q^{30}+2 q^{31}\\
&\quad\quad\quad\quad+q^{32}+q^{33}+q^{34}),\\
p_6(x,q)&=x^4(q^{16}+q^{17}+q^{18}+q^{19}+q^{20}+q^{21})\\
&\quad
+x^5(q^{17}+2 q^{18}+2 q^{19}+3 q^{20}+5 q^{21}+6 q^{22}+5 q^{23}+5 q^{24}+6 q^{25}+5 q^{26}+3 q^{27}\\
&\quad\quad\quad\quad+2 q^{28}+2 q^{29}+q^{30})\\
&\quad
+x^6(q^{20}+2 q^{21}+5 q^{22}+7 q^{23}+8 q^{24}+10 q^{25}+13 q^{26}+14 q^{27}+15 q^{28}+15 q^{29}\\
&\quad\quad\quad\quad+14 q^{30}+13 q^{31}+10 q^{32}+8 q^{33}+7 q^{34}+5 q^{35}+2 q^{36}+q^{37})\\
&\quad
+x^7(q^{24}+2 q^{25}+4 q^{26}+7 q^{27}+10 q^{28}+13 q^{29}+16 q^{30}+19 q^{31}+21 q^{32}+21 q^{33}\\
&\quad\quad\quad\quad+21 q^{34}+21 q^{35}+19 q^{36}+16 q^{37}+13 q^{38}+10 q^{39}+7 q^{40}+4 q^{41}+2 q^{42}\\
&\quad\quad\quad\quad+q^{43})\\
&\quad
+x^8(q^{29}+q^{30}+2 q^{31}+5 q^{32}+7 q^{33}+10 q^{34}+12 q^{35}+14 q^{36}+16 q^{37}+16 q^{38}\\
&\quad\quad\quad\quad+16 q^{39}+16 q^{40}+14 q^{41}+12 q^{42}+10 q^{43}+7 q^{44}+5 q^{45}+2 q^{46}+q^{47}\\
&\quad\quad\quad\quad+q^{48})\\
&\quad
+x^9(q^{36}+q^{37}+2 q^{38}+3 q^{39}+3 q^{40}+4 q^{41}+5 q^{42}+5 q^{43}+5 q^{44}+5 q^{45}+4 q^{46}\\
&\quad\quad\quad\quad+3 q^{47}+3 q^{48}+2 q^{49}+q^{50}+q^{51}),\\
p_{9}(x,q)&=-x^6q^{36}
-x^7(q^{37}+q^{38}+q^{39}+q^{40}+q^{41}+q^{42}+q^{43}+q^{44}+q^{45})\\
&\quad
-x^8(q^{39}+q^{40}+q^{41}+2 q^{42}+2 q^{43}+3 q^{44}+3 q^{45}+4 q^{46}+4 q^{47}+4 q^{48}+3 q^{49}\\
&\quad\quad\quad\quad+3 q^{50}+2 q^{51}+2 q^{52}+q^{53})\\
&\quad
-x^9(q^{43}+q^{44}+q^{45}+2 q^{46}+3 q^{47}+4 q^{48}+6 q^{49}+7 q^{50}+8 q^{51}+8 q^{52}+8 q^{53}\\
&\quad\quad\quad\quad+8 q^{54}+8 q^{55}+6 q^{56}+5 q^{57}+4 q^{58}+2 q^{59}+q^{60}+q^{61})\\
&\quad
-x^{10}(q^{50}+3 q^{51}+4 q^{52}+5 q^{53}+6 q^{54}+8 q^{55}+10 q^{56}+11 q^{57}+13 q^{58}+13 q^{59}\\
&\quad\quad\quad\quad+11 q^{60}+9 q^{61}+8 q^{62}+7 q^{63}+5 q^{64}+3 q^{65}+2 q^{66}+q^{67})\\
&\quad
-x^{11}(q^{55}+q^{56}+q^{57}+3 q^{58}+5 q^{59}+7 q^{60}+8 q^{61}+8 q^{62}+9 q^{63}+10 q^{64}+10 q^{65}\\
&\quad\quad\quad\quad+10 q^{66}+8 q^{67}+6 q^{68}+4 q^{69}+2 q^{70}+2 q^{71}+q^{72})\\
&\quad
-x^{12}(q^{62}+q^{63}+q^{64}+2 q^{65}+2 q^{66}+3 q^{67}+4 q^{68}+4 q^{69}+3 q^{70}+3 q^{71}+3 q^{72}\\
&\quad\quad\quad\quad+2 q^{73}+2 q^{74}+q^{75}),\\
\intertext{and}
p_{12}(x,q)&=x^{12}q^{93} + x^{13}(q^{94}+q^{95}+q^{96}+q^{97}+q^{98}+q^{99})\\
&\quad+ x^{14}(q^{96}+q^{97}+q^{98}+2 q^{99}+2 q^{100}+2 q^{101}+q^{102}+q^{103}+q^{104})\\
&\quad+x^{15}(q^{100}+q^{101}+q^{102}+q^{103}+q^{104}+q^{105}+q^{106}+q^{107}).
\end{align*}}
\end{theorem}

\begin{theorem}
It holds that
\begin{align}
p_0(x,q)G_{\ps_{T_{\mathrm{IV},b}}}(x)&+p_3(x,q)G_{\ps_{T_{\mathrm{IV},b}}}(xq^3)+p_6(x,q)G_{\ps_{T_{\mathrm{IV},b}}}(x q^6)\nonumber\\
&+p_9(x,q)G_{\ps_{T_{\mathrm{IV},b}}}(xq^9)+p_{12}(x,q)G_{\ps_{T_{\mathrm{IV},b}}}(xq^{12})=0,
\end{align}
where
{\footnotesize\begin{align*}
p_0(x,q)&=1+x(q^{5}+q^{6}+q^{7}+q^{8}+q^{9}+q^{10})\\
&\quad+ x^2(q^{11}+q^{12}+q^{13}+2 q^{14}+2 q^{15}+2 q^{16}+q^{17}+q^{18}+q^{19})\\
&\quad+ x^3(q^{19}+q^{20}+q^{21}+q^{22}+q^{23}+q^{24}+q^{25}+q^{26}),\\
p_3(x,q)&=-1-x(q^{2}+q^{3}+q^{4}+q^{5}+q^{6}+q^{7}+q^{8}+q^{9}+q^{10})\\
&\quad
-x^2(q^{5}+2 q^{6}+2 q^{7}+3 q^{8}+3 q^{9}+4 q^{10}+4 q^{11}+4 q^{12}+3 q^{13}+3 q^{14}+2 q^{15}+2 q^{16}\\
&\quad\quad\quad\quad+q^{17}+q^{18}+q^{19})\\
&\quad
-  x^3(q^{8}+q^{9}+2 q^{10}+4 q^{11}+5 q^{12}+6 q^{13}+8 q^{14}+8 q^{15}+8 q^{16}+8 q^{17}+8 q^{18}\\
&\quad\quad\quad\quad+7 q^{19}+6 q^{20}+4 q^{21}+3 q^{22}+2 q^{23}+q^{24}+q^{25}+q^{26})\\
&\quad
-x^4(q^{13}+2 q^{14}+3 q^{15}+5 q^{16}+7 q^{17}+8 q^{18}+9 q^{19}+11 q^{20}+13 q^{21}+13 q^{22}\\
&\quad\quad\quad\quad+11 q^{23}+10 q^{24}+8 q^{25}+6 q^{26}+5 q^{27}+4 q^{28}+3 q^{29}+q^{30})\\
&\quad
-x^5(q^{19}+2 q^{20}+2 q^{21}+4 q^{22}+6 q^{23}+8 q^{24}+10 q^{25}+10 q^{26}+10 q^{27}+9 q^{28}\\
&\quad\quad\quad\quad+8 q^{29}+8 q^{30}+7 q^{31}+5 q^{32}+3 q^{33}+q^{34}+q^{35}+q^{36})\\
&\quad
- x^6(q^{27}+2 q^{28}+2 q^{29}+3 q^{30}+3 q^{31}+3 q^{32}+4 q^{33}+4 q^{34}+3 q^{35}+2 q^{36}+2 q^{37}\\
&\quad\quad\quad\quad+q^{38}+q^{39}+q^{40}),\\
p_6(x,q)&=x^4(q^{20}+q^{21}+q^{22}+q^{23}+q^{24}+q^{25})\\
&\quad
+x^5(q^{22}+2 q^{23}+2 q^{24}+3 q^{25}+5 q^{26}+6 q^{27}+5 q^{28}+5 q^{29}+6 q^{30}+5 q^{31}+3 q^{32}\\
&\quad\quad\quad\quad+2 q^{33}+2 q^{34}+q^{35})\\
&\quad
+x^6(q^{26}+2 q^{27}+5 q^{28}+7 q^{29}+8 q^{30}+10 q^{31}+13 q^{32}+14 q^{33}+15 q^{34}+15 q^{35}\\
&\quad\quad\quad\quad+14 q^{36}+13 q^{37}+10 q^{38}+8 q^{39}+7 q^{40}+5 q^{41}+2 q^{42}+q^{43})\\
&\quad
+x^7(q^{31}+2 q^{32}+4 q^{33}+7 q^{34}+10 q^{35}+13 q^{36}+16 q^{37}+19 q^{38}+21 q^{39}+21 q^{40}\\
&\quad\quad\quad\quad+21 q^{41}+21 q^{42}+19 q^{43}+16 q^{44}+13 q^{45}+10 q^{46}+7 q^{47}+4 q^{48}+2 q^{49}\\
&\quad\quad\quad\quad+q^{50})\\
&\quad
+x^8(q^{37}+q^{38}+2 q^{39}+5 q^{40}+7 q^{41}+10 q^{42}+12 q^{43}+14 q^{44}+16 q^{45}+16 q^{46}\\
&\quad\quad\quad\quad+16 q^{47}+16 q^{48}+14 q^{49}+12 q^{50}+10 q^{51}+7 q^{52}+5 q^{53}+2 q^{54}+q^{55}\\
&\quad\quad\quad\quad+q^{56})\\
&\quad
+x^9(q^{45}+q^{46}+2 q^{47}+3 q^{48}+3 q^{49}+4 q^{50}+5 q^{51}+5 q^{52}+5 q^{53}+5 q^{54}+4 q^{55}\\
&\quad\quad\quad\quad+3 q^{56}+3 q^{57}+2 q^{58}+q^{59}+q^{60}),\\
p_{9}(x,q)&=-x^6q^{42}
-x^7(q^{44}+q^{45}+q^{46}+q^{47}+q^{48}+q^{49}+q^{50}+q^{51}+q^{52})\\
&\quad
-x^8(q^{47}+q^{48}+q^{49}+2 q^{50}+2 q^{51}+3 q^{52}+3 q^{53}+4 q^{54}+4 q^{55}+4 q^{56}+3 q^{57}\\
&\quad\quad\quad\quad+3 q^{58}+2 q^{59}+2 q^{60}+q^{61})\\
&\quad
-x^9(q^{52}+q^{53}+q^{54}+2 q^{55}+3 q^{56}+4 q^{57}+6 q^{58}+7 q^{59}+8 q^{60}+8 q^{61}+8 q^{62}\\
&\quad\quad\quad\quad+8 q^{63}+8 q^{64}+6 q^{65}+5 q^{66}+4 q^{67}+2 q^{68}+q^{69}+q^{70})\\
&\quad
-x^{10}(q^{60}+3 q^{61}+4 q^{62}+5 q^{63}+6 q^{64}+8 q^{65}+10 q^{66}+11 q^{67}+13 q^{68}+13 q^{69}\\
&\quad\quad\quad\quad+11 q^{70}+9 q^{71}+8 q^{72}+7 q^{73}+5 q^{74}+3 q^{75}+2 q^{76}+q^{77})\\
&\quad
-x^{11}(q^{66}+q^{67}+q^{68}+3 q^{69}+5 q^{70}+7 q^{71}+8 q^{72}+8 q^{73}+9 q^{74}+10 q^{75}+10 q^{76}\\
&\quad\quad\quad\quad+10 q^{77}+8 q^{78}+6 q^{79}+4 q^{80}+2 q^{81}+2 q^{82}+q^{83})\\
&\quad
-x^{12}(q^{74}+q^{75}+q^{76}+2 q^{77}+2 q^{78}+3 q^{79}+4 q^{80}+4 q^{81}+3 q^{82}+3 q^{83}+3 q^{84}\\
&\quad\quad\quad\quad+2 q^{85}+2 q^{86}+q^{87}),\\
\intertext{and}
p_{12}(x,q)&=x^{12}q^{105} + x^{13}(q^{107}+q^{108}+q^{109}+q^{110}+q^{111}+q^{112})\\
&\quad+ x^{14}(q^{110}+q^{111}+q^{112}+2 q^{113}+2 q^{114}+2 q^{115}+q^{116}+q^{117}+q^{118})\\
&\quad+x^{15}(q^{115}+q^{116}+q^{117}+q^{118}+q^{119}+q^{120}+q^{121}+q^{122}).
\end{align*}}
\end{theorem}


\section{``Guessing'' the generating functions}

It is, of course, not easy to discover a closed form for each generating function directly from $q$-difference equations obtained in the previous section. However, Andrews' conjecture presented in the introduction shall give us enough clues.

Recall that Andrews' conjecture states as follows.

\begin{conjecture}
Every linked partition ideal $\is$ has a bivariate generating function $G_{\is}(x,q)$ of the form
\begin{equation}\label{eq:And-conj}
\sum_{n_1,\ldots,n_r\ge 0}\frac{(-1)^{L_1(n_1,\ldots,n_r)}q^{Q(n_1,\ldots,n_r)+L_2(n_1,\ldots,n_r)}x^{L_3(n_1,\ldots,n_r)}}{(q^{B_1};q^{A_1})_{n_1}\cdots (q^{B_r};q^{A_r})_{n_r}},
\end{equation}
where $L_1$, $L_2$ and $L_3$ are linear forms in $n_1,\ldots,n_r$ and $Q$ is a quadratic form in $n_1,\ldots,n_r$.
\end{conjecture}

It also appears to be true that some ``nice'' subsets of a linked partition ideal enjoy a generating function of the form \eqref{eq:And-conj}. One may investigate the second Rogers--Ramanujan identity as an example.

Hence, we may search from a number of multi-summations of the form \eqref{eq:And-conj} and compare the series expansions to find suitable candidates.

\begin{theorem}
Let $G_{\ps_{T_{\mathrm{I},1}}}(x,q)$ (resp.~$G_{\ps_{T_{\mathrm{I},2}}}(x,q)$, $G_{\ps_{T_{\mathrm{I},3}}}(x,q)$) denote the generating function of partitions of type I whose smallest part is at least $1$ (resp.~$2$, $3$). We have
\begin{align}
G_{\ps_{T_{\mathrm{I},1}}}(x,q)&=\sum_{n_1,n_2\ge 0}\frac{q^{n_1^2+3n_2^2+3n_1n_2}x^{n_1+2n_2}}{(q;q)_{n_1} (q^3;q^3)_{n_2}},\label{eq:real-gf-t11}\\
G_{\ps_{T_{\mathrm{I},2}}}(x,q)&=\sum_{n_1,n_2\ge 0}\frac{q^{n_1^2+3n_2^2+3n_1n_2+n_1+3n_2}x^{n_1+2n_2}}{(q;q)_{n_1} (q^3;q^3)_{n_2}},\label{eq:real-gf-t12}\\
G_{\ps_{T_{\mathrm{I},3}}}(x,q)&=\sum_{n_1,n_2\ge 0}\frac{q^{n_1^2+3n_2^2+3n_1n_2+2n_1+3n_2}x^{n_1+2n_2}}{(q;q)_{n_1} (q^3;q^3)_{n_2}}.\label{eq:real-gf-t13}
\end{align}
\end{theorem}

\begin{remark}
Here \eqref{eq:real-gf-t11}, \eqref{eq:real-gf-t12} and \eqref{eq:real-gf-t13} are (3.1), (3.10) and (3.14) in \cite{Kur2018}. They correspond to the Kanade--Russell conjectures $I_1$, $I_2$ and $I_3$, respectively.
\end{remark}

\begin{theorem}
Let $G_{\ps_{T_{\mathrm{II},1}}}(x,q)$ (resp.~$G_{\ps_{T_{\mathrm{II},2}}}(x,q)$) denote the generating function of partitions of type II whose smallest part is at least $1$ (resp.~$2$) and let $G_{\ps_{T_{\mathrm{II},a}}}(x,q)$ denote the generating function of partitions of type II where $1$ appears at most once. We have
\begin{align}
G_{\ps_{T_{\mathrm{II},1}}}(x,q)&=\sum_{n_1,n_2\ge 0}\frac{q^{n_1^2+3n_2^2+3n_1n_2-n_2}x^{n_1+2n_2}}{(q;q)_{n_1} (q^3;q^3)_{n_2}},\label{eq:real-gf-t21}\\
G_{\ps_{T_{\mathrm{II},2}}}(x,q)&=\sum_{n_1,n_2\ge 0}\frac{q^{n_1^2+3n_2^2+3n_1n_2+n_1+2n_2}x^{n_1+2n_2}}{(q;q)_{n_1} (q^3;q^3)_{n_2}},\label{eq:real-gf-t22}\\
G_{\ps_{T_{\mathrm{II},a}}}(x,q)&=\sum_{n_1,n_2\ge 0}\frac{q^{n_1^2+3n_2^2+3n_1n_2+2n_2}x^{n_1+2n_2}}{(q;q)_{n_1} (q^3;q^3)_{n_2}}.\label{eq:real-gf-t23}
\end{align}
\end{theorem}

\begin{remark}
Here \eqref{eq:real-gf-t22} is (3.15) in \cite{Kur2018}. It corresponds to the Kanade--Russell conjecture $I_4$.
\end{remark}

\begin{theorem}
Let $G_{\ps_{T_{\mathrm{III},1}}}(x,q)$ (resp.~$G_{\ps_{T_{\mathrm{III},2}}}(x,q)$) denote the generating function of partitions of type III whose smallest part is at least $1$ (resp.~$2$) and let $G_{\ps_{T_{\mathrm{III},a}}}(x,q)$ denote the generating function of partitions of type III where $1$ appears at most once. We have
\begin{align}
&G_{\ps_{T_{\mathrm{III},1}}}(x,q)\nonumber\\
&\quad=\sum_{n_1,n_2,n_3\ge 0}\frac{q^{\frac{n_1^2}{2}+3n_2^2+\frac{9n_3^2}{2}+2 n_1 n_2+6n_2 n_3+3n_3n_1+\frac{n_1}{2}-n_2-\frac{n_3}{2}}x^{n_1+2n_2+3n_3}}{(q;q)_{n_1} (q^2;q^2)_{n_2}  (q^3;q^3)_{n_3}},\label{eq:real-gf-t31}\\
&G_{\ps_{T_{\mathrm{III},2}}}(x,q)\nonumber\\
&\quad=\sum_{n_1,n_2,n_3\ge 0}\frac{q^{\frac{n_1^2}{2}+3n_2^2+\frac{9n_3^2}{2}+2 n_1 n_2+6n_2 n_3+3n_3n_1+\frac{3n_1}{2}+n_2+\frac{5n_3}{2}}x^{n_1+2n_2+3n_3}}{(q;q)_{n_1} (q^2;q^2)_{n_2}  (q^3;q^3)_{n_3}},\label{eq:real-gf-t32}\\
&G_{\ps_{T_{\mathrm{III},a}}}(x,q)\nonumber\\
&\quad=\sum_{n_1,n_2,n_3\ge 0}\frac{q^{\frac{n_1^2}{2}+3n_2^2+\frac{9n_3^2}{2}+2 n_1 n_2+6n_2 n_3+3n_3n_1+\frac{n_1}{2}+n_2+\frac{5n_3}{2}}x^{n_1+2n_2+3n_3}}{(q;q)_{n_1} (q^2;q^2)_{n_2}  (q^3;q^3)_{n_3}}.\label{eq:real-gf-t33}
\end{align}
\end{theorem}

\begin{remark}
Here \eqref{eq:real-gf-t33} is (47) (corrected: in its numerator, the last term of the exponent of $q$ should read $4k$ instead of $3k$) in \cite{KR2018}. It corresponds to the Kanade--Russell conjecture $I_5$.
\end{remark}

\begin{theorem}
Let $G_{\ps_{T_{\mathrm{IV},1}}}(x,q)$ denote the generating function of partitions of type IV whose smallest part is at least $1$, let $G_{\ps_{T_{\mathrm{IV},a}}}(x,q)$ denote the generating function of partitions of type IV where $1$ appears at most once and let $G_{\ps_{T_{\mathrm{IV},b}}}(x,q)$ denote the generating function of partitions of type IV where the smallest part is at least $2$ with $2$ appearing at most once. We have
\begin{align}
&G_{\ps_{T_{\mathrm{IV},1}}}(x,q)\nonumber\\
&\quad=\sum_{n_1,n_2,n_3\ge 0}\frac{q^{\frac{n_1^2}{2}+3n_2^2+\frac{9n_3^2}{2}+2 n_1 n_2+6n_2 n_3+3n_3n_1+\frac{n_1}{2}-n_2+\frac{n_3}{2}}x^{n_1+2n_2+3n_3}}{(q;q)_{n_1} (q^2;q^2)_{n_2}  (q^3;q^3)_{n_3}},\label{eq:real-gf-t41}\\
&G_{\ps_{T_{\mathrm{IV},a}}}(x,q)\nonumber\\
&\quad=\sum_{n_1,n_2,n_3\ge 0}\frac{q^{\frac{n_1^2}{2}+3n_2^2+\frac{9n_3^2}{2}+2 n_1 n_2+6n_2 n_3+3n_3n_1+\frac{n_1}{2}+n_2+\frac{n_3}{2}}x^{n_1+2n_2+3n_3}}{(q;q)_{n_1} (q^2;q^2)_{n_2}  (q^3;q^3)_{n_3}},\label{eq:real-gf-t42}\\
&G_{\ps_{T_{\mathrm{IV},b}}}(x,q)\nonumber\\
&\quad=\sum_{n_1,n_2,n_3\ge 0}\frac{q^{\frac{n_1^2}{2}+3n_2^2+\frac{9n_3^2}{2}+2 n_1 n_2+6n_2 n_3+3n_3n_1+\frac{3n_1}{2}+3n_2+\frac{7n_3}{2}}x^{n_1+2n_2+3n_3}}{(q;q)_{n_1} (q^2;q^2)_{n_2}  (q^3;q^3)_{n_3}}.\label{eq:real-gf-t43}
\end{align}
\end{theorem}

\begin{remark}
Here \eqref{eq:real-gf-t43} is (51) in \cite{KR2018}. It corresponds to the Kanade--Russell conjecture $I_6$.
\end{remark}

In the above theorems, we rediscover six generating function identities proved in \cite{KR2018} and \cite{Kur2018} and obtain six new identities. We will provide an approach to prove these identities in the next section with the help of computer algebra.

\begin{remark}
	It is, of course, fine to discover the above sum-like generating functions by trial and error with this tedious work left to a computer. But sometimes human observation might reduce the workload. Let us use \eqref{eq:real-gf-t11} as an example. If we write
	\begin{equation*}
	G_{\ps_{T_{\mathrm{I},1}}}(x,q)=\sum_{M\ge 0}g_{\ps_{T_{\mathrm{I},1}}}(M)x^M,
	\end{equation*}
	then the $q$-difference equation in Theorem \ref{th:q-diff-T11} gives the first several expressions of $g_{\ps_{T_{\mathrm{I},1}}}(M)$:
	\begin{align*}
	g_{\ps_{T_{\mathrm{I},1}}}(0)&=1,\\
	g_{\ps_{T_{\mathrm{I},1}}}(1)&=\frac{q}{1-q},\\
	g_{\ps_{T_{\mathrm{I},1}}}(2)&=\frac{q^3+q^4+q^6}{(1-q^2)(1-q^3)},\\
	g_{\ps_{T_{\mathrm{I},1}}}(3)&=\frac{q^7}{(1-q)(1-q^2)(1-q^3)},\\
	g_{\ps_{T_{\mathrm{I},1}}}(4)&=\frac{q^{12}+q^{15}+q^{17}+q^{18}-q^{19}+q^{20}-q^{21}}{(1-q)(1-q^3)(1-q^4)(1-q^6)}.
	\end{align*}
	Recall that the sum-like generating function is
	\begin{equation}\label{eq:and-gf}
	\sum_{n_1,\ldots,n_r\ge 0}\frac{(-1)^{L_1(n_1,\ldots,n_r)}q^{Q(n_1,\ldots,n_r)+L_2(n_1,\ldots,n_r)}x^{L_3(n_1,\ldots,n_r)}}{(q^{B_1};q^{A_1})_{n_1}\cdots (q^{B_r};q^{A_r})_{n_r}}.
	\end{equation}
	We observe that the numerator of $g_{\ps_{T_{\mathrm{I},1}}}(2)$ has more than one term. Hence, the linear equation $L_3(n_1,\ldots,n_r)=2$ might have multiple nonnegative solutions $(n_1,\ldots,n_r)$. It is fair to guess that $L_3$ looks like $n_1+n_2+\cdots$ or $n_1+2n_2+\cdots$. Also, the denominators of $g_{\ps_{T_{\mathrm{I},1}}}(M)$ indicate that there might be terms like $(q;q)_{n}$ and $(q^3;q^3)_{n}$ in the denominator of the summand in \eqref{eq:and-gf}. Hence, one may first try multi-summations like
	$$\sum_{n_1,n_2\ge 0}\frac{(-1)^{L_1(n_1,n_2)}q^{Q(n_1,n_2)+L_2(n_1,n_2)}x^{n_1+n_2}}{(q;q)_{n_1} (q^{3};q^{3})_{n_2}}$$
	or
	$$\sum_{n_1,n_2\ge 0}\frac{(-1)^{L_1(n_1,n_2)}q^{Q(n_1,n_2)+L_2(n_1,n_2)}x^{n_1+2n_2}}{(q;q)_{n_1} (q^{3};q^{3})_{n_2}}.$$
	If these expressions fail to be a candidate, then one could continue to modify them and carry on the searching procedure. However, it should be emphasized that in this remark we do not intend to assert that the sum-like generating function must contain some particular ``magical'' exponents and bases.
\end{remark}

\section{Computer algebra assistance}

Proofs of the generating function identities in the previous section can be carried out by the same procedure. We only demonstrate \eqref{eq:real-gf-t11} as an instance.

\subsection{The main idea}

Let us write
\begin{equation}\label{eq:T11-new}
G_{\ps_{T_{\mathrm{I},1}}}(x,q)=\sum_{M\ge 0}g_{\ps_{T_{\mathrm{I},1}}}(M)x^M,
\end{equation}
where $g_{\ps_{T_{\mathrm{I},1}}}(M)\in\mathbb{Q}(q)$. We can translate the $q$-difference equation in Theorem \ref{th:q-diff-T11} to a recurrence of $g_{\ps_{T_{\mathrm{I},1}}}(M)$.

\begin{definition}
Let $\mathbb{K}=\mathbb{Q}(q)$ with $q$ transcendental. A sequence $(a_n)$ in $\mathbb{K}$ is called \textit{$q$-holonomic} if there exist $p,p_0,\ldots,p_r\in\mathbb{K}[x]$, not all zero, such that
$$p_0(q^n)a_n+p_1(q^n)a_{n+1}+\cdots+p_r(q^n)a_{n+r}=p(q^n).$$
\end{definition}

Hence, the sequence $g_{\ps_{T_{\mathrm{I},1}}}(M)$ is $q$-holonomic.

On the other hand, if we write
\begin{equation}\label{eq:T11-NewNew}
\sum_{n_1,n_2\ge 0}\frac{q^{n_1^2+3n_2^2+3n_1n_2}x^{n_1+2n_2}}{(q;q)_{n_1} (q^3;q^3)_{n_2}}=\sum_{M\ge 0}\tilde{g}_{\ps_{T_{\mathrm{I},1}}}(M)x^M,
\end{equation}
we may also find a recurrence relation satisfied by $\tilde{g}_{\ps_{T_{\mathrm{I},1}}}(M)$. Hence, $\tilde{g}_{\ps_{T_{\mathrm{I},1}}}(M)$ is also $q$-holonomic.

A result of Kauers and Koutschan \cite{KK2009} states that if two sequences $(a_n)$ and $(b_n)$ are $q$-holonomic, so is their linear combination $(\alpha a_n+\beta b_n)$. Hence, we may find a recurrence relation satisfied by $g_{\ps_{T_{\mathrm{I},1}}}(M)-\tilde{g}_{\ps_{T_{\mathrm{I},1}}}(M)$. As long as $g_{\ps_{T_{\mathrm{I},1}}}(M)-\tilde{g}_{\ps_{T_{\mathrm{I},1}}}(M)=0$ for enough initial cases, we are safe to say that this difference is identical to $0$ for all $M$ and hence arrive at the desired generating function identity.

\subsection{Two \textit{Mathematica} packages}

To proceed with our proof, we require two \textit{Mathematica} packages: \texttt{qMultiSum} \cite{Rie2003} and \texttt{qGeneratingFunctions} \cite{KK2009}. These packages along with their instructions can be found on the webpage of Research Institute for Symbolic Computation (RISC) of Johannes Kepler University.\footnote{See \url{https://www3.risc.jku.at/research/combinat/software/ergosum/index.html}.}

To begin with, we load the two packages after installing them.
\begin{lstlisting}[language=Mathematica]
<<RISC`qMultiSum`
<<RISC`qGeneratingFunctions`
\end{lstlisting}

\subsection{Recurrence for $g_{\ps_{T_{\mathrm{I},1}}}(M)$}

For the polynomials $p_{3i}(x,q)$ ($i=0,\ldots,3$) defined in Theorem \ref{th:q-diff-T11}, we write
$$p_{3i}(x,q)=\sum_{j=0}^{J_{3i}}p_{3i,j}(q)x^j.$$

Then with \eqref{eq:T11-new}, one may rewrite \eqref{eq:q-diff-T11} as
\begin{align*}
0&=\sum_{i=0}^3 p_{3i}(x,q)G_{\ps_{T_{\mathrm{I},1}}}(xq^{3i})\\
&=\sum_{i=0}^3\sum_{j=0}^{J_{3i}}\sum_{m\ge 0}p_{3i,j}(q)g_{\ps_{T_{\mathrm{I},1}}}(m) q^{3im}x^{m+j}\\
&=\sum_{M\ge 0}\sum_{i=0}^3 \sum_{m=\max(0,M-J_{3i})}^M q^{3im}p_{3i,M-m}(x,q)g_{\ps_{T_{\mathrm{I},1}}}(m) x^M.
\end{align*}
Hence, for all $M\ge 0$,
\begin{align}
\sum_{i=0}^3 \sum_{m=\max(0,M-J_{3i})}^M q^{3im}p_{3i,M-m}(x,q)g_{\ps_{T_{\mathrm{I},1}}}(m) = 0,
\end{align}
from which we see that $g_{\ps_{T_{\mathrm{I},1}}}(M)$ ($M\ge 1$) is uniquely determined by $g_{\ps_{T_{\mathrm{I},1}}}(0)$. It is also trivial that $g_{\ps_{T_{\mathrm{I},1}}}(0)=1$.

In particular, for $M\ge 0$, we have the following recurrence
\begin{align}
0&=g_{\ps_{T_{\mathrm{I},1}}}(M)\left((q^{28}+q^{30})q^{9M}\right)\nonumber\\
&\quad+g_{\ps_{T_{\mathrm{I},1}}}(M+1)\left((q^{19}+q^{21})q^{6(M+1)}+q^{27}q^{9(M+1)}\right)\nonumber\\
&\quad+g_{\ps_{T_{\mathrm{I},1}}}(M+2)\left((q^{14}+q^{15}+q^{16}+q^{17}+q^{18})q^{6(M+2)}\right)\nonumber\\
&\quad+g_{\ps_{T_{\mathrm{I},1}}}(M+3)\left(-(q^7+q^9+q^{10}+q^{12})q^{3(M+3)}+(q^{11}+q^{13})q^{6(M+3)}\right)\nonumber\\
&\quad+g_{\ps_{T_{\mathrm{I},1}}}(M+4)\left(-(q^3+q^4+q^5+2q^6+q^7+q^8+q^9)q^{3(M+4)}\right)\nonumber\\
&\quad+g_{\ps_{T_{\mathrm{I},1}}}(M+5)\left((q^4+q^6)-(q+q^2+q^3+q^4+q^6)q^{3(M+5)}\right)\nonumber\\
&\quad+g_{\ps_{T_{\mathrm{I},1}}}(M+6)\left(1-q^{3(M+6)}\right).\label{eq:T11-rec1}
\end{align}

\subsection{Recurrence for $\tilde{g}_{\ps_{T_{\mathrm{I},1}}}(M)$}

Notice that for $M\ge 0$,
$$\tilde{g}_{\ps_{T_{\mathrm{I},1}}}(M)=\sum_{n\le \frac{M}{2}}\frac{q^{(M-2n)^2+3n^2+3n(M-2n)}}{(q;q)_{M-2n}(q^3;q^3)_n}.$$

The recurrence satisfied by $\tilde{g}_{\ps_{T_{\mathrm{I},1}}}(M)$ can be computed automatically by the \texttt{qMultiSum} package with the following codes:
\begin{lstlisting}[language=Mathematica]
ClearAll[M];
summand = q^(3n^2+(M-2n)^2+3n(M-2 n))/(qPochhammer[q,q,M-2n] qPochhammer[q^3,q^3,n]);
stru = qFindStructureSet[summand, {M}, {n}, {1}, {2}, {2}, qProtocol -> True]
rec = qFindRecurrence[summand, {M}, {n}, {1}, {2}, {2}, qProtocol -> True, StructSet -> stru[[1]]]
sumrec = qSumRecurrence[rec]
\end{lstlisting}

This gives us, for $M\ge 0$,
{\footnotesize\begin{align}
0&=\tilde{g}_{\ps_{T_{\mathrm{I},1}}}(M)q^{9M+24}(1+2q^2+q^4+q^{3M+14})\nonumber\\
&\quad+\tilde{g}_{\ps_{T_{\mathrm{I},1}}}(M+1)q^{6M+21}(1+2q^2+q^4-q^{3M+8}-q^{3M+10}+q^{3M+11}+q^{3M+13}+q^{3M+14})\nonumber\\
&\quad+\tilde{g}_{\ps_{T_{\mathrm{I},1}}}(M+2)q^{6M+22}(1+q^2)(1+q^2+q^3+q^4+q^{3M+12})\nonumber\\
&\quad-\tilde{g}_{\ps_{T_{\mathrm{I},1}}}(M+3)q^{3M+12}(1+q^2)(1-q+q^2)(1+q+q^2+q^3+q^{3M+12})\nonumber\\
&\quad-\tilde{g}_{\ps_{T_{\mathrm{I},1}}}(M+4)q^{3M+12}(1-q+q^2)(1+q+q^2)(1+q+q^2+q^3+q^{3M+13})\nonumber\\
&\quad+\tilde{g}_{\ps_{T_{\mathrm{I},1}}}(M+5)(1-q^{3M+15})(1+2q^2+q^4+q^{3M+11}).\label{eq:T11-rec2}
\end{align}}

\subsection{Recurrence for $g_{\ps_{T_{\mathrm{I},1}}}(M)-\tilde{g}_{\ps_{T_{\mathrm{I},1}}}(M)$}

Finally, we deduce the recurrence for $g_{\ps_{T_{\mathrm{I},1}}}(M)-\tilde{g}_{\ps_{T_{\mathrm{I},1}}}(M)$ from \eqref{eq:T11-rec1} and \eqref{eq:T11-rec2}. This can be accomplished by the \texttt{QREPlus} function of the \texttt{qGeneratingFunctions} package.

We need the following codes, in which \texttt{sumrec1} records the recurrence relation for $g_{\ps_{T_{\mathrm{I},1}}}(M)$ and \texttt{sumrec2} records the recurrence relation for $\tilde{g}_{\ps_{T_{\mathrm{I},1}}}(M)$.
\begin{lstlisting}[language=Mathematica]
ClearAll[M];
sumrec1 = {SUM[M] ((q^(28)+q^(30))q^(9M))
 + SUM[M+1] ((q^(19)+q^(21))q^(6(M+1))+q^(27)q^(9(M+1)))
 + SUM[M+2] ((q^(14)+q^(15)+q^(16)+q^(17)+q^(18))q^(6(M+2)))
 + SUM[M+3] (-(q^7+q^9+q^(10)+q^(12))q^(3(M+3))+(q^(11)+q^(13))q^(6(M+3)))
 + SUM[M+4] (-(q^3+q^4+q^5+2q^6+q^7+q^8+q^9)q^(3(M+4)))
 + SUM[M+5] ((q^4+q^6)-(q+q^2+q^3+q^4+q^6)q^(3(M+5)))
 + SUM[M+6] (1-q^(3(M+6)))
== 0};
sumrec2 = {SUM[M] q^(9M+24) (1+2q^2+q^4+q^(3M+14))
 + SUM[M+1] q^(6M+21) (1+2q^2+q^4-q^(3M+8)-q^(3M+10)+q^(3M+11)+q^(3M+13)+q^(3M+14))
 + SUM[M+2] q^(6M+22) (1+q^2) (1+q^2+q^3+q^4+q^(3M+12))
 - SUM[M+3] q^(3M+12) (1+q^2) (1-q+q^2) (1+q+q^2+q^3+q^(3M+12))
 - SUM[M+4] q^(3M+12) (1-q+q^2) (1+q+q^2) (1+q+q^2+q^3+q^(3M+13))
 + SUM[M+5] (1-q^(3M+15)) (1+2q^2+q^4+q^(3M+11))
== 0};
QREPlus[sumrec1, sumrec2, SUM[M]]
\end{lstlisting}

The output gives us an order six recurrence. Hence, to show
$$g_{\ps_{T_{\mathrm{I},1}}}(M)=\tilde{g}_{\ps_{T_{\mathrm{I},1}}}(M)$$
for all $M\ge 0$, it suffices to show that the equality holds for $M=0,\ldots,5$. This can be checked easily.

We therefore arrive at
$$G_{\ps_{T_{\mathrm{I},1}}}(x,q)=\sum_{n_1,n_2\ge 0}\frac{q^{n_1^2+3n_2^2+3n_1n_2}x^{n_1+2n_2}}{(q;q)_{n_1} (q^3;q^3)_{n_2}}.$$

\subsection{Other identities}

Similar to \eqref{eq:T11-new} and \eqref{eq:T11-NewNew}, let us write
$$G_{\ps_{T_{*}}}(x,q)=\sum_{M\ge 0}g_{\ps_{T_{*}}}(M)x^M$$
and the multi-summations on the right hand sides of \eqref{eq:real-gf-t12}--\eqref{eq:real-gf-t43} as
$$\sum_{M\ge 0}\tilde{g}_{\ps_{T_{*}}}(M)x^M,$$
where ``$*$'' may be ``$\mathrm{I},2$'', ``$\mathrm{I},3$'', etc. We list the orders of recurrences satisfied by $g_{\ps_{T_{*}}}(M)$, $\tilde{g}_{\ps_{T_{*}}}(M)$ and $g_{\ps_{T_{*}}}(M)-\tilde{g}_{\ps_{T_{*}}}(M)$ in Table \ref{ta:order} for the reader's convenience.

{\footnotesize\begin{table}[ht]\caption{Orders of recurrences satisfied by $g_{\ps_{T_{*}}}(M)$, $\tilde{g}_{\ps_{T_{*}}}(M)$ and $g_{\ps_{T_{*}}}(M)-\tilde{g}_{\ps_{T_{*}}}(M)$}\label{ta:order}
\centering
\begin{tabular}{cccccccccccc}
\hline
$*$ & $\mathrm{I},2$ & $\mathrm{I},3$ & $\mathrm{II},1$ & $\mathrm{II},2$ & $\mathrm{II},a$ & $\mathrm{III},1$ & $\mathrm{III},2$ & $\mathrm{III},a$ & $\mathrm{IV},1$ & $\mathrm{IV},a$ & $\mathrm{IV},b$\\
$g_{\ps_{T_{*}}}$ & 6 & 6 & 6 & 6 & 6 & 15 & 15 & 15 & 15 & 15 & 15\\
$\tilde{g}_{\ps_{T_{*}}}$ & 5 & 5 & 5 & 5 & 5 & 4 & 4 & 4 & 4 & 4 & 4\\
$g_{\ps_{T_{*}}}-\tilde{g}_{\ps_{T_{*}}}$ & 6 & 6 & 6 & 6 & 6 & 15 & 15 & 15 & 15 & 15 & 15\\
\hline
\end{tabular}
\end{table}}

\section{Closing remarks}

In a very recent paper of Bringmann, Jennings-Shaffer and Mahlburg \cite{BJM2018}, along with other results, the Kanade--Russell conjectures $I_5$ and $I_6$ were proved. Here the analytic forms of $I_5$ and $I_6$ read respectively as
\begin{align}
G_{\ps_{T_{\mathrm{III},a}}}(1,q)&=\sum_{n_1,n_2,n_3\ge 0}\frac{q^{\frac{n_1^2}{2}+3n_2^2+\frac{9n_3^2}{2}+2 n_1 n_2+6n_2 n_3+3n_3n_1+\frac{n_1}{2}+n_2+\frac{5n_3}{2}}}{(q;q)_{n_1} (q^2;q^2)_{n_2}  (q^3;q^3)_{n_3}}\nonumber\\
&=\frac{1}{ ( q, q^3, q^4, q^6, q^7, q^{10}, q^{11} ; q^{12})_\infty },\\[10pt]
G_{\ps_{T_{\mathrm{IV},b}}}(1,q)&=\sum_{n_1,n_2,n_3\ge 0}\frac{q^{\frac{n_1^2}{2}+3n_2^2+\frac{9n_3^2}{2}+2 n_1 n_2+6n_2 n_3+3n_3n_1+\frac{3n_1}{2}+3n_2+\frac{7n_3}{2}}}{(q;q)_{n_1} (q^2;q^2)_{n_2}  (q^3;q^3)_{n_3}}\nonumber\\
&=\frac{1}{ ( q^2, q^3, q^5, q^6, q^7, q^8, q^{11} ; q^{12})_\infty }.
\end{align}

The authors of \cite{BJM2018} cleverly reformulated $G_{\ps_{T_{\mathrm{III},a}}}(1,q)$ and $G_{\ps_{T_{\mathrm{IV},b}}}(1,q)$ and then added a new parameter so that the new bivariate generating functions satisfy simpler $q$-difference equations, from which the authors deduced the above identities.

Recall the standard notation for basic hypergeometric series:
$${}_{r+1}\phi_r\left(\begin{matrix} a_0,a_1,a_2\ldots,a_r\\ b_1,b_2,\ldots,b_r \end{matrix}; q, z\right):=\sum_{n\ge 0}\frac{(a_0;q)_n(a_1;q)_n\cdots(a_r;q)_n}{(q;q)_n(b_1;q)_n\cdots (b_r;q)_n} z^n.$$

Following the proofs of (1.15) and (1.16) in \cite{BJM2018}, one may prove the following identities with no difficulty.

\begin{theorem}
We have
\begin{align}
G_{\ps_{T_{\mathrm{III},1}}}(1,q)&=\sum_{n_1,n_2,n_3\ge 0}\frac{q^{\frac{n_1^2}{2}+3n_2^2+\frac{9n_3^2}{2}+2 n_1 n_2+6n_2 n_3+3n_3n_1+\frac{n_1}{2}-n_2-\frac{n_3}{2}}}{(q;q)_{n_1} (q^2;q^2)_{n_2}  (q^3;q^3)_{n_3}}\nonumber\\
&=(-q;q)_\infty (-q^3;q^6)_\infty\; {}_{2}\phi_1\left(\begin{matrix} q^{-1},q\\ q^2 \end{matrix}; q^6, -q^3\right),\\[10pt]
G_{\ps_{T_{\mathrm{III},2}}}(1,q)&=\sum_{n_1,n_2,n_3\ge 0}\frac{q^{\frac{n_1^2}{2}+3n_2^2+\frac{9n_3^2}{2}+2 n_1 n_2+6n_2 n_3+3n_3n_1+\frac{3n_1}{2}+n_2+\frac{5n_3}{2}}}{(q;q)_{n_1} (q^2;q^2)_{n_2}  (q^3;q^3)_{n_3}}\nonumber\\
&=(-q^2;q)_\infty (-q^3;q^6)_\infty\; {}_{2}\phi_1\left(\begin{matrix} q,q^5\\ q^8 \end{matrix}; q^6, -q^3\right),\\[10pt]
G_{\ps_{T_{\mathrm{IV},1}}}(1,q)&=\sum_{n_1,n_2,n_3\ge 0}\frac{q^{\frac{n_1^2}{2}+3n_2^2+\frac{9n_3^2}{2}+2 n_1 n_2+6n_2 n_3+3n_3n_1+\frac{n_1}{2}-n_2+\frac{n_3}{2}}}{(q;q)_{n_1} (q^2;q^2)_{n_2}  (q^3;q^3)_{n_3}}\nonumber\\
&=(-q;q)_\infty (-q^3;q^6)_\infty\; {}_{2}\phi_1\left(\begin{matrix} q^{-1},q\\ q^4 \end{matrix}; q^6, -q^3\right),\\[10pt]
G_{\ps_{T_{\mathrm{IV},a}}}(1,q)&=\sum_{n_1,n_2,n_3\ge 0}\frac{q^{\frac{n_1^2}{2}+3n_2^2+\frac{9n_3^2}{2}+2 n_1 n_2+6n_2 n_3+3n_3n_1+\frac{n_1}{2}+n_2+\frac{n_3}{2}}}{(q;q)_{n_1} (q^2;q^2)_{n_2}  (q^3;q^3)_{n_3}}\nonumber\\
&=(-q;q)_\infty (-q^3;q^6)_\infty\; {}_{2}\phi_1\left(\begin{matrix} q,q^5\\ q^4 \end{matrix}; q^6, -q^3\right).
\end{align}
\end{theorem}

Note that we shall use a refinement of Proposition 2.4 in \cite{BJM2018}, the proof of which comes from a slight modification of the original proof of Bringmann, Jennings-Shaffer and Mahlburg.

\begin{proposition}
Suppose that $A(x)=\sum_{n\ge0}\alpha_nx^n$ has positive radius of convergence
and $A(x)$ satisfies
\begin{align*}
A(x)&=\left(1+q^a+x^2q^b+x^2q^c\right)A\left(xq^3\right)\nonumber\\
&\quad -q^a\left(1+x^2q^{b+c-a-d+6}\right)\left(1+x^2q^{d}\right)A\left(xq^6\right),
\end{align*}
where $a\not\in 3\mathbb{Z}$ if $a\le-6$.
Then
\begin{align}
A(x)&=\alpha_0\left(-x^2q^{d-6};q^6\right)_\infty\;\sum_{n\ge0}\frac{ \left( q^{b-d+6},q^{c-d+6};q^6 \right)_n (-1)^nq^{(d-6)n}  }{\left( q^{6},q^{a+6};q^6 \right)_n}x^{2n}\nonumber\\
&\quad+\alpha_1 \left(-x^2q^{d-6};q^6\right)_\infty\;\sum_{n\ge0}\frac{ \left( q^{b-d+9},q^{c-d+9};q^6 \right)_n (-1)^nq^{(d-6)n} }{\left( q^{9}, q^{a+9};q^6 \right)_n}x^{2n+1}.
\end{align}
\end{proposition}

\subsection*{Acknowledgements}

We would like to thank George Andrews for helpful discussions and for sharing his conjecture. We also want to thank the referees for their detailed comments on an earlier version of this paper.

\bibliographystyle{amsplain}

\end{document}